\documentclass{cpamart1}     
\received{April 2018}       
\yearofpublication{2019}
\volume{000}
\startingpage{1}                      

\authorheadline{Boumal, Voroninski, Bandeira}
\titleheadline{Guarantees for Burer--Monteiro factorizations of smooth SDPs}

\usepackage{url}

\usepackage{mathtools}

\DeclarePairedDelimiter{\floorr}{\lfloor}{\rfloor}

\usepackage{amsfonts}
\usepackage{amsthm}  
\usepackage{amssymb}

\usepackage{color}
\usepackage{bm} 
\usepackage{dsfont} 

\usepackage{graphicx}

\usepackage{mathrsfs} 

\newcommand{\citep}{\cite}

\RequirePackage[colorlinks,citecolor=blue,urlcolor=blue,linkcolor=blue]{hyperref}

\usepackage{mathptmx}

\newtheorem{theorem}{Theorem}[section]
\newtheorem{lemma}[theorem]{Lemma}
\newtheorem{corollary}[theorem]{Corollary}
\newtheorem{proposition}[theorem]{Proposition}
\theoremstyle{definition}
\newtheorem{definition}[theorem]{Definition}
\theoremstyle{remark}
\newtheorem{remark}[theorem]{Remark}

\newtheorem{assumption}[theorem]{Assumption}

\newcommand{\argmin}[1]{\underset{#1}{\operatorname{argmin}}}

\newcommand{\transpose}{^\top\! }

\newcommand{\inner}[2]{\left\langle{#1},{#2}\right\rangle}
\newcommand{\innersmall}[2]{\langle{#1},{#2}\rangle}
\newcommand{\innerbig}[2]{\big\langle{#1},{#2}\big\rangle}

\newcommand{\trace}{\mathrm{Tr}}
\newcommand{\Trace}{\mathrm{Tr}}
\newcommand{\spann}{\mathrm{span}}
\newcommand{\im}{\mathrm{im}}

\newcommand{\sbdop}{\operatorname{sbd}}
\newcommand{\sbd}[1]{\sbdop\!\left({#1}\right)}

\newcommand{\Proj}{\mathrm{Proj}}

\newcommand{\St}{\mathrm{St}}

\newcommand{\T}{\mathrm{T}}
\newcommand{\N}{\mathrm{N}}

\newcommand{\Snn}{{\mathbb{S}^{n\times n}}}

\newcommand{\Spp}{{\mathbb{S}^{p\times p}}}

\newcommand{\Rnn}{{\mathbb{R}^{n\times n}}}
\newcommand{\Rnp}{{\mathbb{R}^{n\times p}}}

\newcommand{\Rpd}{{\mathbb{R}^{p\times d}}}

\newcommand{\Rpp}{\mathbb{R}^{p\times p}}

\newcommand{\Rp}{{\mathbb{R}^{p}}}
\newcommand{\Rk}{{\mathbb{R}^{k}}}

\newcommand{\reals}{{\mathbb{R}}}

\newcommand{\Rn}{{\mathbb{R}^n}}
\newcommand{\Rm}{{\mathbb{R}^m}}
\newcommand{\Rq}{{\mathbb{R}^q}}

\newcommand{\grad}{\mathrm{grad}\,}
\newcommand{\Hess}{\mathrm{Hess}\,}
\newcommand{\vecc}{\mathrm{vec}}

\newcommand{\diag}{\mathrm{diag}}
\newcommand{\D}{\mathrm{D}}

\newcommand{\calM}{\mathcal{M}}
\newcommand{\calN}{\mathcal{N}}
\newcommand{\calC}{\mathcal{C}}
\newcommand{\calE}{\mathcal{E}}
\newcommand{\calF}{\mathcal{F}}
\newcommand{\calL}{\mathcal{L}}

\newcommand{\calO}{\mathcal{O}}
\newcommand{\calT}{\mathcal{T}}
\newcommand{\calA}{\mathcal{A}}
\newcommand{\Id}{\operatorname{Id}} 
\newcommand{\rank}{\operatorname{rank}}
\newcommand{\nulll}{\operatorname{null}}

\newcommand{\floor}[1]{\lfloor #1 \rfloor}

\newcommand{\lambdamin}{\lambda_\mathrm{min}}

\newcommand{\st}{\textrm{ subject to }}

\begin{document}                        


\title{Deterministic guarantees for Burer--Monteiro factorizations of smooth semidefinite programs}

\author{Nicolas Boumal}{Mathematics Department and Program in Applied and Computational Mathematics, Princeton University}
\author{Vladislav Voroninski}{Helm.ai}
\author{Afonso S.\ Bandeira}{Department of Mathematics and Center for Data Science,\\Courant Institute of Mathematical Sciences, New York University}





\begin{abstract}

We consider semidefinite programs (SDPs) with equality constraints. The variable to be optimized is a positive semidefinite matrix $X$ of size $n$. Following the Burer--Monteiro approach, we optimize a factor $Y$ of size $n \times p$ instead, such that $X = YY\transpose$. This ensures positive semidefiniteness at no cost and can reduce the dimension of the problem if $p$ is small, but results in a non-convex optimization problem with a quadratic cost function and quadratic equality constraints in $Y$. In this paper, we show that if the set of constraints on $Y$ regularly defines a smooth manifold, then, despite non-convexity, first- and second-order necessary optimality conditions are also sufficient, provided $p$ is large enough. For smaller values of $p$, we show a similar result holds for almost all (linear) cost functions. Under those conditions, a global optimum $Y$ maps to a global optimum $X = YY\transpose$ of the SDP. We deduce old and new consequences for SDP relaxations of the generalized eigenvector problem, the trust-region subproblem and quadratic optimization over several spheres, as well as for the Max-Cut and Orthogonal-Cut SDPs which are common relaxations in stochastic block modeling and synchronization of rotations.\\

\end{abstract}

\maketitle   






\section{Introduction}

We consider semidefinite programs (SDPs) of the form
\begin{align}
f^\star = 
\min_{X\in\Snn} \inner{C}{X} \quad \st \quad \calA(X) = b, \ X \succeq 0,
\tag{SDP}
\label{eq:SDP}
\end{align}
where $\Snn$ is the set of real symmetric matrices of size $n$, $C \in \Snn$ is the cost matrix, $\inner{C}{X} = \trace(C\transpose X)$, $\calA \colon \Snn \to \Rm$ is a linear operator capturing $m$ equality constraints with right-hand side $b\in\Rm$, and the variable $X$ is symmetric, positive semidefinite.
Let $A_1, \ldots, A_m \in \Snn$ be the constraint matrices such that $\calA(X)_i = \inner{A_i}{X}$, and let
\begin{align}
	\calC & = \left\{ X \in \Snn : \calA(X) = b \textrm{ and } X \succeq 0 \right\}
	\label{eq:calC}
\end{align}
be the search space of~\eqref{eq:SDP}, assumed non empty.

Interior point methods solve~\eqref{eq:SDP} in polynomial time~\citep{nesterov1994interior}. In practice however, for $n$ beyond a few thousands, such algorithms run out of memory (and time), prompting research for alternative solvers. Crucially, if $\calC$ is compact, then~\eqref{eq:SDP} admits a global optimum of rank at most~$r$, where $\frac{r(r+1)}{2} \leq m$~\citep{pataki1998rank,barvinok1995problems}---we review this fact in Section~\ref{sec:geometryconvex}.
Thus, if one restricts $\calC$ to matrices of rank at most $p$ with $\frac{p(p+1)}{2} \geq m$, the  optimal value remains unchanged. This restriction is easily enforced by factorizing $X = YY\transpose$ where $Y$ has size $n\times p$, yielding a quadratically constrained quadratic program:
\begin{align}
\min_{Y\in\Rnp} \inner{CY}{Y} \quad \textrm{ subject to } \quad \calA(YY\transpose) = b.
\tag{P}
\label{eq:P}
\end{align}
In general, \eqref{eq:P} is non-convex because its search space
\begin{align}
	\calM_{p} & = \left\{ Y \in \Rnp : \calA(YY\transpose\,) = b \right\}
	\label{eq:calM}
\end{align}
is non-convex. (When $p$ is clear from context or unimportant, we just write $\calM$.)

Non-convexity makes it a priori unclear how to solve~\eqref{eq:P}. Still, the benefits are that $\calM$ requires no conic constraint and can be lower dimensional than $\calC$.
This has motivated Burer and Monteiro~\citep{sdplr,burer2005local} to try to solve~\eqref{eq:P} using local optimization methods, with surprisingly good results. They developed theory in support of this observation (details below). About their results, 
Burer and Monteiro write:
\begin{quote}
	``\emph{How large must we take $p$ so that the local minima of~\eqref{eq:P} are guaranteed to map to global minima of~\eqref{eq:SDP}? Our theorem asserts that we need only\footnote{The condition on $p$ and $m$ is slightly, but inconsequentially, different in~\citep{burer2005local}.} $\frac{p(p+1)}{2} > m$ (with the important caveat that positive-dimensional faces of~\eqref{eq:SDP} which are `flat' with respect to the objective function can harbor non-global local minima).}''
	\begin{flushright}
		--- End of Section 3 in~\citep{burer2005local}, mutatis mutandis.
	\end{flushright}
\end{quote}
The caveat---the existence or non-existence of non-global local optima, or their potentially adverse effect for local optimization algorithms---was not further discussed. How mild this caveat really is (as stated) is hard to gauge, considering $\calC$ can have a continuum of faces.

\subsection*{Contributions}

In this paper, we identify settings where the non-convexity of~\eqref{eq:P} is benign, in the sense that second-order necessary optimality conditions are sufficient for global optimality---an unusual property for a non-convex problem. This paper extends a previous conference paper by the same authors~\citep{boumal2016bmapproach}. Our core assumption is as follows.
\begin{assumption}\label{assu:M}
	For a given $p$ such that $\calM$~\eqref{eq:calM} is non-empty, constraints on~\eqref{eq:SDP} defined by $A_1, \ldots, A_m \in \Snn$ and $b \in \Rm$ satisfy at least one of the following:
	\begin{enumerate}
		\item[a.] $\{A_1Y, \ldots, A_mY\}$ are linearly independent in $\Rnp$ for all $Y \in \calM$; or
		\item[b.] $\{A_1Y, \ldots, A_mY\}$ span a subspace of constant dimension in $\Rnp$ for all $Y$ in an open neigh\-bor\-hood of $\calM$ in $\Rnp$.
	\end{enumerate}
	In either case, let $m'$ denote the dimension of the space spanned by $\{A_1Y, \ldots, A_mY\}$. (By assumption, $m'$ is independent of the choice of $Y\in\calM$.)
\end{assumption}
Under Assumption~\ref{assu:M}, $\calM$ is a smooth manifold, which is why we say such an~\eqref{eq:SDP} is \emph{smooth}. Furthermore, if the assumption holds for several values of $p$, then $m'$ is the same for all. Formal statements follow; proofs are in Appendix~\ref{apdx:assuforallp}.
\begin{proposition}\label{prop:submanifold}
	Under Assumption~\ref{assu:M}, $\calM$ is an embedded submanifold of $\Rnp$ of dimension $np - m'$.
\end{proposition}
\begin{proposition}\label{prop:assuforallp}
	If Assumption~\ref{assu:M} holds for some $p$, it holds for all $p' \leq p$ such that $\calM_{p'}$ is non-empty. Furthermore, if Assumption~\ref{assu:M}a holds for $p = n$, then it holds for all $p'$ such that $\calM_{p'}$ is non-empty. In both cases, $m'$ is independent of $p$.
\end{proposition}
Examples of SDPs satisfying Assumption~\ref{assu:M} are detailed in Section~\ref{sec:applications} (they all satisfy Assumption~\ref{assu:M}a for $p = n$). The assumption itself is further discussed in Section~\ref{sec:discussionassumptions}.
Our first main result is as follows, where $\rank\calA$ can be replaced by $m$ if preferred. Optimality conditions are derived in Section~\ref{sec:geometryandoptimconditions}.
\begin{theorem}\label{thm:mastersmallp}
	Let $p$ be such that $\frac{p(p+1)}{2} > \rank\calA$ and such that Assumption~\ref{assu:M} holds. For almost any cost matrix $C \in \Snn$, if $Y \in \calM$ satisfies first- and second-order necessary optimality conditions for~\eqref{eq:P}, then $Y$ is globally optimal and $X = YY\transpose$ is globally optimal for~\eqref{eq:SDP}.
\end{theorem}
%
%
%
The proof combines two intermediate results (Proposition~\ref{prop:rankdeficientY} and Lemma~\ref{lem:criticalptsrankdeficient} below):
\begin{enumerate}
	\item If $Y$ is \emph{column-rank deficient} and satisfies first- and second-order necessary optimality conditions for~\eqref{eq:P}, then it is globally optimal and $X=YY\transpose$ is optimal for~\eqref{eq:SDP}; and
	\item If $\frac{p(p+1)}{2} > \rank \calA$, then, for almost all $C$, every $Y$ which satisfies first-order necessary optimality conditions is column-rank deficient.
\end{enumerate}
The first step is a variant of well-known results~\citep{sdplr,burer2005local,journee2010low}. The second step is new and crucial, as it allows to formally exclude the existence of spurious local optima, thus resolving the caveat raised by Burer and Monteiro generically in $C$.

Theorem~\ref{thm:mastersmallp} is a statement about the optimization problem itself, not about specific algorithms.
If $\calC$ is compact, then so is $\calM$ and known algorithms for optimization on manifolds converge to \emph{second-order critical points},\footnote{Points which satisfy first- and second-order necessary optimality conditions. Compactness of $\calC$ ensures a minimum is attained in~\eqref{eq:P}, hence also that second-order critical points exist.} regardless of initialization~\citep{boumal2016globalrates}.
%
%
%
%
%
Thus, provided $p$ is large enough, for almost any cost matrix $C$, such algorithms generate sequences which converge to global optima of~\eqref{eq:P}. Each iteration requires a polynomial number of arithmetic operations.

In practice, the algorithm is stopped after a finite number of iterations, at which point one can only guarantee approximate satisfaction of first- and second-order necessary optimality conditions. Ideally, this should lead to a statement of approximate optimality. We are only able to make that statement for large values of $p$.
We state this result informally here, and give a precise statement in Corollary~\ref{cor:approxsocpapproxoptimalpnplusone} below.
\begin{theorem}[Informal]\label{thm:stableinformal}
	Assume $\calC$ is compact and Assumption~\ref{assu:M} holds for $p = n+1$. Then, for any cost matrix $C \in \Snn$, if $Y \in \calM_{n+1}$ approximately satisfies first- and second-order necessary optimality conditions for~\eqref{eq:P}, then it is approximately globally optimal and $X = YY\transpose$ is approximately globally optimal for~\eqref{eq:SDP}, in terms of attained cost value.
\end{theorem}

Theorem~\ref{thm:mastersmallp} does not exclude the possibility that a zero-measure subset of cost matrices $C$ may pose difficulties. Theorem~\ref{thm:stableinformal} does apply for all cost matrices, but requires a large value of $p$. A complementary result in this paper, which comes with a more geometric proof, constitutes a refinement of the caveat raised by Burer and Monteiro~\citep{burer2005local} in the excerpt quoted above. It states that a suboptimal second-order critical point $Y$ must map to a face $\calF_{YY\transpose}$ of the convex search space $\calC$ whose dimension is large (rather than just positive) when $p$ itself is large. The facial structure of $\calC$ is discussed in Section~\ref{sec:geometryconvex}. The following is a consequence of Corollary~\ref{cor:YoptimalifSpsd} and Theorem~\ref{thm:eigS} below.
\begin{theorem}\label{thm:deterministic}
	Let Assumption~\ref{assu:M} hold for some $p$.
	Let $Y\in\calM$ be a second-order critical point of~\eqref{eq:P}. If $\rank(Y) < p$, or if $\rank(Y) = p$ and $\dim \calF_{YY\transpose} < \frac{p(p+1)}{2} - m' + p$, then $Y$ is globally optimal for~\eqref{eq:P} and $X = YY\transpose$ is globally optimal for~\eqref{eq:SDP}.
\end{theorem}
Combining this theorem with bounds on the dimension of faces of $\calC$ allows us to conclude the optimality of second-order critical points for \emph{all} cost matrices $C$, with bounds on $p$ that are smaller than $n$.
Implications of these theorems for examples of SDPs are treated in Section~\ref{sec:applications}, including the trust-region subproblem, Max-Cut and Orthogonal-Cut.

\subsection*{Notation}

$\Snn$ is the set of real, symmetric matrices of size $n$. A symmetric matrix $X$ is positive semidefinite ($X \succeq 0$) if and only if $u\transpose X u \geq 0$ for all $u\in\Rn$. For matrices $A, B$, the standard Euclidean inner product is $\inner{A}{B} = \trace(A\transpose B)$. The associated (Frobenius) norm is $\|A\| = \sqrt{\inner{A}{A}}$. $\Id$ is the identity operator and $I_n$ is the identity matrix of size $n$. The variable $m' \leq m$ is defined in Assumption~\ref{assu:M}. The adjoint of $\calA$ is $\calA^*$, such that $\calA^*(\nu) = \nu_1 A_1 + \cdots + \nu_m A_m$.

\section{Geometry and optimality conditions}\label{sec:geometryandoptimconditions}

We first discuss the smooth geometry of~\eqref{eq:P} and the convex geometry of~\eqref{eq:SDP}, as well as optimality conditions for both.

\subsection{For the non-convex problem~\eqref{eq:P}}
\label{sec:geometryriemannian}

Endow $\Rnp$ with the classical Euclidean metric $\inner{U_1}{U_2} = \Trace(U_1\transpose U_2^{})$, corresponding to the Frobenius norm: $\|U\|^2 = \inner{U}{U}$. As stated in Proposition~\ref{prop:submanifold}, under Assumption~\ref{assu:M} for a given $p$, the search space $\calM$ of~\eqref{eq:P} defined in~\eqref{eq:calM} is a submanifold of $\Rnp$ of dimension $\dim \calM = np - m'$. 
Furthermore,
the tangent space to $\calM$ at $Y$ is a subspace of $\Rnp$ obtained by linearizing the equality constraints.
\begin{lemma}
	Under Assumption~\ref{assu:M}, the tangent space at $Y$ to $\calM$, $\T_Y\calM$, obeys
	\begin{align}
	\T_Y\calM & = \left\{ \dot Y \in \Rnp : \calA(\dot Y Y\transpose + Y \dot Y\transpose\,) = 0 \right\} \nonumber\\
	& = \left\{ \dot Y \in \Rnp : \innersmall{A_iY}{\dot Y} = 0 \textrm{ for } i = 1, \ldots, m \right\}.
	\label{eq:tangentspace} 
	\end{align}
\end{lemma}
\begin{proof}
	By definition, $\dot Y \in \Rnp$ is a tangent vector to $\calM$ at $Y$ if and only if there exists a curve $\gamma \colon \reals \to \calM$ such that $\gamma(0) = Y$ and $\dot \gamma(0) = \dot Y$, where $\dot \gamma$ is the derivative of $\gamma$. Then, $\calA(\gamma(t)\gamma(t)\transpose) = b$ for all $t$. Differentiating on both sides yields $\calA(\dot \gamma(t) \gamma(t)\transpose + \gamma(t) \dot \gamma(t)\transpose) = 0$. Evaluating at $t = 0$ confirms
	$\T_Y\calM$ is included in the subspace~\eqref{eq:tangentspace}.
	To conclude, use the fact that both subspaces have the same dimension under Assumption~\ref{assu:M}, by Proposition~\ref{prop:submanifold}.
\end{proof}
Each tangent space is equipped with a restriction of the metric $\inner{\cdot}{\cdot}$, thus making $\calM$ a \emph{Riemannian} submanifold of $\Rnp$.
From~\eqref{eq:tangentspace}, it is clear that the $A_iY$ span the normal space at $Y$:
\begin{align}
	\N_Y\calM & = \spann\{ A_1 Y, \ldots, A_m Y \}.
	\label{eq:normalspace}
\end{align}
%
%
An important tool is the orthogonal projector $\Proj_Y \colon \Rnp \to \T_Y\calM$:
\begin{align}
	\Proj_Y Z & = \argmin{\dot Y \in \T_Y\calM} \  \| \dot Y - Z \|.
	\label{eq:Proj}
\end{align}
We have the following lemma to characterize it.
\begin{lemma}\label{lem:Proj}
	Under Assumption~\ref{assu:M}, the orthogonal projector is given by:
	\begin{align*}
		\Proj_Y Z & = Z - \calA^*\!\left( G^\dagger \calA(ZY\transpose) \right) Y,
	\end{align*}
	where $\calA^* \colon \Rm \to \Snn$ is the adjoint of $\calA$, $G = G(Y)$ is a Gram matrix defined by $G_{ij} = \inner{A_iY}{A_jY}$, and $G^\dagger$ denotes the Moore--Penrose pseudo-inverse of $G$. Furthermore, if $Y \mapsto Z(Y)$ is differentiable in an open neighborhood of $\calM$ in $\Rnp$, then $Y \mapsto \Proj_Y Z(Y)$ is differentiable at all $Y$ in $\calM$.
\end{lemma}
\begin{proof}
Orthogonal projection is along the normal space, so that $\Proj_Y Z \in \T_Y\calM$ and $Z - \Proj_Y Z \in \N_Y\calM$~\eqref{eq:normalspace}. From the latter we infer there exists $\mu\in\Rm$ such that
\begin{align*}
	Z - \Proj_Y Z = \sum_{i = 1}^m \mu_i A_i Y = \calA^*(\mu) Y,
\end{align*}
since the adjoint of $\calA$ is $\calA^*(\mu) = \mu_1 A_1 + \cdots + \mu_m A_m$ by definition.
Multiply on the right by $Y\transpose$ and apply $\calA$ to obtain
%
%
%
%
\begin{align*}
	\calA(ZY\transpose) = \calA( \calA^*(\mu) Y Y\transpose ),
\end{align*}
where we used $\calA(\Proj_Y(Z) Y\transpose) = 0$ since $\Proj_Y(Z) \in \T_Y\calM$. The right-hand side expands into
\begin{align*}
	\calA( \calA^*(\mu) Y Y\transpose )_i = \inner{A_i}{\sum_{j=1}^m \mu_j A_j YY\transpose} = \sum_{j=1}^m \inner{A_i Y}{A_j Y} \mu_j = (G\mu)_i.
\end{align*}
Thus, any $\mu$ satisfying $G\mu = \calA(ZY\transpose)$ will do.
Without loss of generality, we pick the smallest norm solution: $\mu = G^\dagger \calA(ZY\transpose)$.
The function $Y \mapsto G^\dagger$ is continuous and differentiable at $Y\in\calM$ provided $G$ has constant rank in an open neighborhood of $Y$ in $\Rnp$~\cite[Thm.~4.3]{golub1973differentiation}, which is the case under Assumption~\ref{assu:M}.
\end{proof}


Problem~\eqref{eq:P} minimizes 
\begin{align}
	g(Y) & = \inner{CY}{Y}
	\label{eq:g}
\end{align}
over $\calM$, where $g$ is defined over $\Rnp$. Its classical (Euclidean) gradient at $Y$ is $\nabla g(Y) = 2CY$. The Riemannian gradient of $g$ at $Y$, $\grad g(Y)$, is defined as the unique tangent vector at $Y$ such that, for all tangent $\dot Y$, 
$
\innersmall{\grad g(Y)}{\dot Y} = \innersmall{\nabla g(Y)}{\dot Y}. 
$
This is given by the projection of the classical gradient onto the tangent space~\citep[eq.\,(3.37)]{AMS08}:
\begin{align*}
	\grad g(Y) & = \Proj_Y\left(\nabla g(Y)\right) = 2\,\Proj_Y \left(C Y \right) = 2\left(C - \calA^*\!\left( G^\dagger \calA(CYY\transpose) \right)\right)Y.
\end{align*}
This motivates the definition of $S$ as follows, with $G_{ij} = \inner{A_i Y}{A_j Y}$:
\begin{align}
	S = S(Y) = S(YY\transpose) & = C - \calA^*\!\left( \mu \right), & \textrm{ with } & & \mu & = G^\dagger \calA(CYY\transpose).
	\label{eq:S}
\end{align}
This is indeed well defined since $G_{ij}$ is a function of $YY\transpose$.
We get a convenient formula for the gradient:
\begin{align}
	\grad g(Y) & = 2SY.
	\label{eq:gradg}
\end{align}
In the sequel, $S$ will play a major role.

Turning toward second-order derivatives, the Riemannian Hessian of $g$ at $Y$ is a symmetric operator on the tangent space at $Y$ obtained as the projection of the derivative of the Riemannian gradient vector field~\citep[eq.\,(5.15)]{AMS08}. The latter is indeed differentiable owing to Lemma~\ref{lem:Proj}. With $\D$ denoting classical Fr\'echet differentiation, writing $S = S(Y)$ and $\dot S = \D( Y \mapsto S(Y))(Y)[\dot Y]$,
\begin{align}
	\Hess g(Y)[\dot Y] & = \Proj_Y\!\left( \D\grad g(Y)[\dot Y] \right) 
	= 2\Proj_Y\!\big( \dot S Y + S \dot Y \big)
	= 2\Proj_Y\!\big( S\dot Y \big).
	\label{eq:Hessg}
\end{align}
The projection of $\dot S Y$ vanishes because $\dot S = \calA^*(\nu)$ for some $\nu \in \Rm$ so that $\dot S Y = \sum_{i = 1}^m \nu_i A_i Y$ is in the normal space at $Y$~\eqref{eq:normalspace}.

These differentials are relevant for their role in necessary optimality conditions of~\eqref{eq:P}.
\begin{definition}\label{def:critical}
	$Y\in\calM$ is a \emph{(first-order) critical point} for~\eqref{eq:P} if
	\begin{align}
		\frac{1}{2} \grad g(Y) = SY = 0,
		\label{eq:firstordercondition}
	\end{align}
	where $S$ is a function of $Y$~\eqref{eq:S}. If furthermore $\Hess g(Y) \succeq 0$, that is (using the fact that $\Proj_Y$ is self-adjoint),
	\begin{align}
		\forall \dot Y \in \T_Y\calM, \quad \frac{1}{2} \innersmall{\dot Y}{\Hess g(Y)[\dot Y]} = \innersmall{\dot Y}{S\dot Y} \geq 0,
		\label{eq:secondorderconditionbis}
	\end{align}
	then $Y$ is a \emph{second-order critical point} for~\eqref{eq:P}.
\end{definition}
\begin{proposition}
	Under Assumption~\ref{assu:M}, all local (and global) minima of~\eqref{eq:P} are second-order critical points.
\end{proposition}
\begin{proof}
	These are standard necessary optimality conditions on manifolds, see~\citep[Rem. 4.2 and Cor. 4.2]{yang2012optimality}.
\end{proof}
Thus, the central role of $S$ in necessary optimality conditions for the non-convex problem is clear. Its role for the convex problem is elucidated next.



\subsection{For the convex problem~\eqref{eq:SDP}} \label{sec:geometryconvex}

The search space of~\eqref{eq:SDP} is the convex set $\calC$ defined in~\eqref{eq:calC}, assumed non-empty.
Geometry-wise, we are primarily interested in the facial structure of $\calC$~\citep[\S18]{rockafellar1997convex}.
\begin{definition}\label{def:face}
	A \emph{face} of $\calC$ is a convex subset $\calF$ of $\calC$ such that every (closed) line segment in $\calC$ with a relative interior point in $\calF$ has both endpoints in $\calF$. The empty set and $\calC$ itself are faces of $\calC$. 
\end{definition}
For example, the non-empty faces of a cube are its vertices, edges, facets and the cube itself.
By~\citep[Thm.\,18.2]{rockafellar1997convex}, the collection of relative interiors of the non-empty faces forms a partition of $\calC$ (the relative interior of a singleton is the singleton). That is, each $X\in\calC$ is in the relative interior of exactly one face of $\calC$, called $\calF_X$. 
The dimension of a face is the dimension of the lowest dimensional affine subspace which contains that face.
Of particular interest are the zero-dimensional faces of $\calC$ (singletons).
\begin{definition}
	$X\in\calC$ is an \emph{extreme point} of $\calC$ if $\dim \calF_X = 0$.
\end{definition}
In other words, $X$ is extreme if it does not lie on an open line segment included in $\calC$. If $\calC$ is compact, it is the convex hull of its extreme points~\citep[Cor.\,18.5.1]{rockafellar1997convex}.
Of importance to us, if $\calC$ is compact, \eqref{eq:SDP} always attains its minimum at one of its extreme points since the linear cost function of~\eqref{eq:SDP} is (a fortiori) concave~\citep[Cor.\,32.3.2]{rockafellar1997convex}.
%
The faces of $\calC$	 can be described explicitly as follows. The proof is in Appendix~\ref{apdx:faces}.
\begin{proposition} \label{prop:faces} 
	Let $X\in\calC$ have rank $p$ and let $\calF_X$ be its associated face (that is, $X$ is in the relative interior of $\calF_X$.) Then, with $Y\in\calM_p$ such that $X = YY\transpose$, 
	\begin{align}
		\calF_X & = \left\{ X' = Y(I_p + A)Y\transpose : A \in \ker \calL_X \textrm{ and } I_p + A \succeq 0\right\},
	\label{eq:calF}
	\end{align}
	where $\calL_X \colon \Spp \to \Rm$ is defined by:
	\begin{align}
		\calL_X(A) & = \calA(YAY\transpose) = \begin{pmatrix}
	\inner{Y\transpose A_1 Y}{A}, \ldots, \inner{Y\transpose A_m Y}{A}
	\end{pmatrix}\transpose.
	\label{eq:calL}
	\end{align}
\end{proposition}
Thus, the dimension of $\calF_X$ is the dimension of the kernel of $\calL_X$. Since the dimension of $\Spp$ is $\frac{p(p+1)}{2}$ and $\rank(\calL_X) \leq m'$,
the rank-nullity theorem gives a lower bound: 
\begin{align}
	\dim \calF_X = \frac{p(p+1)}{2} - \rank \calL_X \geq \frac{p(p+1)}{2} - m'.
	\label{eq:dimF}
\end{align}
%
For extreme points, $\dim \calF_X = 0$; then, $\frac{p(p+1)}{2} = \rank \calL_X \leq m'$. Solving for $p$ (the rank of $X$) shows extreme points have small rank, namely,
\begin{align}
	\dim \calF_X = 0 \implies \rank(X) \leq p^* \triangleq \frac{\sqrt{8m' + 1} - 1}{2}.
	\label{eq:patakirank}
\end{align}
Since~\eqref{eq:SDP} attains its minimum at an extreme point for compact $\calC$, we recover the known fact that one of the optima has rank at most $p^*$. This approach to proving that statement is well known~\citep[Thm.\,2.1]{pataki1998rank}.

Optimality conditions for~\eqref{eq:SDP} are easily stated once $S$~\eqref{eq:S} is introduced---it acts as a dual certificate, known in closed form owing to the underlying smooth geometry of $\calM$. We need a first
general fact about SDPs (Assumption~\ref{assu:M} is not required.) 
\begin{proposition}\label{prop:XoptifSpsd}
	Let $X\in\calC$ and let $S = C - \calA^*(\nu)$ for some $\nu\in\Rm$ (as is the case in~\eqref{eq:S} for example). If $S \succeq 0$ and $\inner{S}{X} = 0$, then $X$ is optimal for~\eqref{eq:SDP}.
\end{proposition}
\begin{proof}
	First, use $S \succeq 0$: for any $X' \in \calC$, since $X' \succeq 0$ and $\calA(X) = \calA(X')$,
	\begin{align*}
		0 \leq \innersmall{S}{X'} = \innersmall{C}{X'} - \innersmall{\calA^*(\nu)}{X'} 
		 = \innersmall{C}{X'} - \innersmall{\nu}{\calA(X)}.
	\end{align*}
	Concentrating on the last term, use $\inner{S}{X} = 0$:
	\begin{align*}
		\inner{\nu}{\calA(X)} = \inner{\calA^*(\nu)}{X} = \inner{C}{X} - \inner{S}{X} = \inner{C}{X}.
	\end{align*}
	Hence, $\inner{C}{X} \leq \inner{C}{X'}$, which shows $X$ is optimal.
\end{proof}
Since~\eqref{eq:SDP} is a relaxation of~\eqref{eq:P}, this leads to a corollary of prime importance.
\begin{corollary} \label{cor:YoptimalifSpsd}
	Let Assumption~\ref{assu:M} hold for some $p$.
	If $Y$ is a critical point for~\eqref{eq:P} as defined by~\eqref{eq:firstordercondition} and $S$~\eqref{eq:S} is positive semidefinite, then $X = YY\transpose$ is globally optimal for~\eqref{eq:SDP} and $Y$ is globally optimal for~\eqref{eq:P}.
\end{corollary}
\begin{proof}
	Since $Y$ is a critical point, $SY = 0$; thus, $\inner{S}{X} = 0$ and Proposition~\ref{prop:XoptifSpsd} applies.
\end{proof}
%
A converse of Proposition~\ref{prop:XoptifSpsd} holds under additional conditions which are satisfied by all examples in Section~\ref{sec:applications}. Thus, for those cases, for a critical point $Y$, $YY\transpose$ is optimal if and only if $S$ is positive semidefinite. We state it here for completeness (this result is not needed in the sequel.)
\begin{proposition}\label{prop:XoptimalImpliesSpsd}
	Let $X\in\calC$ be a global optimum of~\eqref{eq:SDP} and assume strong duality holds.
	Let Assumption~\ref{assu:M}a hold with $p = \rank(X)$.
	Then, $S \succeq 0$ and $\inner{S}{X} = 0$, where $S = S(X)$ is as in~\eqref{eq:S}.
\end{proposition}
\begin{proof}
	Consider the dual of~\eqref{eq:SDP}:
	\begin{align}
	\max_{\nu \in \Rm} \inner{b}{\nu} \st C- \calA^*(\nu) \succeq 0.
	\tag{DSDP}
	\label{eq:DSDP}
	\end{align}
	Since we assume strong duality and $X$ is optimal,
	there exists $\nu$ optimal for the dual such that $\inner{C}{X} = \inner{b}{\nu}$. Using $\inner{b}{\nu} = \inner{\calA(X)}{\nu} = \inner{X}{\calA^*(\nu)}$, this implies
	\begin{align*}
	0 = \inner{C}{X} - \inner{b}{\nu} = \inner{C-\calA^*(\nu)}{X}.
	\end{align*}
	Since both $C-\calA^*(\nu)$ and $X$ are positive semidefinite, $(C-\calA^*(\nu))X = 0$. As a result, by definition of $\mu$ and $G$~\eqref{eq:S},
	\begin{align*}
		\mu = G^\dagger \calA(CX) = G^\dagger \calA( \calA^*(\nu) X ) = G^\dagger G \nu = \nu,
	\end{align*}
	where we used $G^\dagger = G^{-1}$ under Assumption~\ref{assu:M}a and
	\begin{align*}
		(G\nu)_i & = \sum_j G_{ij} \nu_j = \sum_j \inner{A_i}{A_j X} \nu_j 
		 = \inner{A_i}{\calA^*(\nu)X} = \calA(\calA^*(\nu)X)_i.
	\end{align*}
	%
	%
	Thus, $S = C - \calA^*(\mu) = C - \calA^*(\nu)$ has the desired properties. This concludes the proof, and shows uniqueness of the dual certificate.
	%
\end{proof}

\section{Optimality of second-order critical points}



We aim to show that second-order critical points of~\eqref{eq:P} are global optima, provided $p$ is sufficiently large. To this end, we first recall a known result about \emph{rank-deficient} second-order critical points.\footnote{Optimality of rank deficient \emph{local optima} is shown (under different assumptions) in~\citep{burer2005local,journee2010low}, with the proof in~\citep{journee2010low} actually only requiring second-order criticality.}
\begin{proposition} \label{prop:rankdeficientY} 
	Let Assumption~\ref{assu:M} hold for some $p$ and
	let $Y\in\calM$ be a second-order critical point for~\eqref{eq:P}. If $\rank(Y) < p$, then $S(Y) \succeq 0$ so that $Y$ is globally optimal for~\eqref{eq:P} and so is $X = YY\transpose$ for~\eqref{eq:SDP}.
\end{proposition}
\begin{proof}
	The proof parallels the one in~\citep{journee2010low}.
	By Corollary~\ref{cor:YoptimalifSpsd}, it is sufficient to show that $S = S(Y)$~\eqref{eq:S} is positive semidefinite.
	Since $\rank(Y) < p$, there exists $z \in \Rp$ such that $z \neq 0$ and $Yz = 0$. Furthermore, for all $x\in\Rn$, the matrix $\dot Y = xz\transpose$ is such that $Y\dot Y\transpose = 0$. In particular, $\dot Y$ is a tangent vector at $Y$~\eqref{eq:tangentspace}. Since $Y$ is second-order critical, inequality~\eqref{eq:secondorderconditionbis} holds, and here simplifies to:
	\begin{align*}
		0 \leq \innerbig{\dot Y}{S \dot Y} = \innerbig{xz\transpose}{Sxz\transpose\,} = \|z\|^2 \cdot x\transpose S x.
	\end{align*}
	This holds for all $x\in\Rn$. 
	Thus, $S$ is positive semidefinite.
\end{proof}
\begin{corollary}
	Let Assumption~\ref{assu:M} hold for some  $p \geq n$. Then, any second-order critical point $Y\in\calM$ of~\eqref{eq:P} is globally optimal, and $X = YY\transpose$ is globally optimal for~\eqref{eq:SDP}.
\end{corollary}
\begin{proof}
	For $p > n$ (with $p = n+1$ being the most interesting case), points in $\calM$ are necessarily column-rank deficient, so that the corollary follows from Proposition~\ref{prop:rankdeficientY}. For $p = n$, if $Y$ is rank deficient, use the same proposition. Otherwise, $Y$ is invertible and $SY = 0$~\eqref{eq:firstordercondition} implies $S = 0$, which is a fortiori positive semidefinite.
	By~\eqref{eq:S}, this only happens if $C = \calA^*(\mu)$ for some $\mu$, in which case the cost function $\inner{C}{X} = \inner{\calA^*(\mu)}{X} = \inner{\mu}{b}$ is constant over $\calC$. 
\end{proof}

In this paper, we aim to secure optimality of second-order critical points for $p$ less than $n$. As indicated by Proposition~\ref{prop:rankdeficientY}, the sole concern in that respect is the possible existence of \emph{full-rank} second-order critical points. We first give a result which excludes the existence of full-rank first-order critical points (thus, a fortiori of second-order critical points) for \emph{almost all} cost matrices $C$, provided $p$ is sufficiently large. 
The argument is by dimensionality counting. 
\begin{lemma} \label{lem:criticalptsrankdeficient}
	Let $p$ be such that $\frac{p(p+1)}{2} > \rank\calA$ and such that Assumption~\ref{assu:M} holds.
	Then, for almost all $C$, all critical points of~\eqref{eq:P} are column-rank deficient.
\end{lemma}
\begin{proof}
	Let $Y \in \calM$ be a critical point for~\eqref{eq:P}. By the definition of $S(Y) = C - \calA^*(\mu(Y))$~\eqref{eq:S} and the first-order condition $S(Y)Y=0$~\eqref{eq:firstordercondition}, we have 
	\begin{align}
		\rank Y \leq \nulll(C-\calA^*(\mu(Y))) \leq \max_{\nu\in\Rm} \nulll(C-\calA^*(\nu)),
		\label{eq:baseinequality}
	\end{align}
	where $\nulll$ denotes the nullity (dimension of the kernel). This first step in the proof is inspired by~\cite[Thm.\,3]{wen2013orthogonality}.
	If the right-hand side evaluates to $\ell$, then there exists $\nu$ and $M = C-\calA^*(\nu)$ such that $\nulll(M) = \ell$. Writing $C = M + \calA^*(\nu)$, we find that
	\begin{align}
		C & \in \calN_\ell + \im(\calA^*),
	\end{align}
	where $\calN_\ell$ denotes the set of symmetric matrices of size $n$ with nullity $\ell$ and the $+$ is a set-sum. The set $\calN_\ell$ has dimension 
	\begin{align}
		\dim \calN_\ell = \frac{n(n+1)}{2} - \frac{\ell(\ell+1)}{2}.
	\end{align}
	Assume the right-hand side of~\eqref{eq:baseinequality} evaluates to $p$ or more. Then, a fortiori,
	\begin{align}
		C \in \bigcup_{\ell = p,\ldots,n} \calN_\ell + \im(\calA^*).
		\label{eq:CbelongsUnion}
	\end{align}
	The set on the right-hand side contains all ``bad'' matrices $C$, that is, those for which~\eqref{eq:baseinequality} offers no information about the rank of $Y$. The dimension of that set is bounded as follows, using the fact that the dimension of a finite union is at most the maximal dimension, and the dimension of a finite sum of sets is at most the sum of the set dimensions:
	\begin{align*}
	\dim\left( \bigcup_{\ell = p,\ldots,n} \calN_\ell + \im(\calA^*) \right) & \leq \dim\left( \calN_p + \im(\calA^*) \right) \\
	& \leq \frac{n(n+1)}{2} - \frac{p(p+1)}{2} + \rank\calA.
	\end{align*}
	Since $C\in\Snn$ lives in a space of dimension $\frac{n(n+1)}{2}$, almost no $C$ verifies~\eqref{eq:CbelongsUnion} if
	\begin{align*}
	\frac{n(n+1)}{2} - \frac{p(p+1)}{2} + \rank\calA < \frac{n(n+1)}{2}.
	\end{align*}
	Hence, if $\frac{p(p+1)}{2} > \rank\calA$, for almost all $C$, critical points have $\rank(Y) < p$.
\end{proof}
Theorem~\ref{thm:mastersmallp} follows as an easy corollary of Proposition~\ref{prop:rankdeficientY} and Lemma~\ref{lem:criticalptsrankdeficient}.

In order to make a statement valid for \emph{all} $C$, we further explore the implications of second-order criticality on the definiteness of $S$. For large $p$ (though still smaller than $n$), we expect full-rank second-order critical points should indeed be optimal. The intuition is as follows. If $Y\in\calM$ is a second-order critical point of rank $p$, then, by~\eqref{eq:firstordercondition}, $SY = 0$ which implies $S$ has a kernel of dimension at least $p$. Furthermore, by~\eqref{eq:secondorderconditionbis}, $S$ has ``positive curvature'' along directions in $\T_Y\calM$, whose dimension grows with $p$. Overall, the larger $p$, the more conditions force $S$ to have nonnegative eigenvalues. The main concern is to avoid double counting, as the two conditions are redundant along certain directions: this is where the facial structure of $\calC$ comes into play.

The following theorem refines this intuition.
We use $\otimes$ for Kronecker products and $\vecc$ to vectorize a matrix by stacking its columns on top of each other, so that $\vecc(AXB) = (B\transpose \otimes A)\vecc(X)$. 
A real number $a$ is rounded down as $\floor{a}$.
\begin{theorem} \label{thm:eigS}
	Let $p$ be such that Assumption~\ref{assu:M} holds.
	Let $Y\in\calM$ be a second-order critical point for~\eqref{eq:P}. The matrix $X = YY\transpose$ belongs to the relative interior of the face $\calF_X$~\eqref{eq:calF}. If $\rank(Y) = p$, then $S = S(X)$~\eqref{eq:S} has at most
	\begin{align}
		\floorr*{\frac{\dim \calF_X - \Delta}{p}}
		\label{eq:dimFdeltap}
	\end{align}
	negative eigenvalues, where
	\begin{align}
		\Delta = \frac{p(p+1)}{2} - m'.
		\label{eq:Delta}
	\end{align}
	In particular, if $\dim \calF_X < \Delta + p$, then $S$ is positive semidefinite and both $X$ and $Y$ are globally optimal.
\end{theorem}
\begin{proof}
	Consider the subspace $\vecc(\T_Y\calM)$ of vectorized tangent vectors at $Y$: it has dimension $k \triangleq \dim \calM$.
	Pick $U\in\reals^{np\times k}$ with columns forming an orthonormal basis for that subspace: $U\transpose U = I_k$.
	Then, $U\transpose(I_p \otimes S)U$ has the same spectrum as $\frac{1}{2} \Hess g(Y)$. Indeed, for all $\dot Y \in \T_Y\calM$ there exists $x \in \Rk$ such that $\vecc(\dot Y) = Ux$, and, by~\eqref{eq:secondorderconditionbis},
	\begin{align*}
		\frac{1}{2} \innersmall{\dot Y}{\Hess g(Y)[\dot Y]} & = \innersmall{\dot Y}{S \dot Y} = \innersmall{Ux}{(I_p \otimes S)Ux} = \innersmall{x}{U\transpose(I_p \otimes S)Ux}.
	\end{align*}
	In particular, $U\transpose(I_p \otimes S)U$ is positive semidefinite since $Y$ is second-order critical. 
	
	Let $V\in\reals^{np\times p^2}, V\transpose V = I_{p^2},$ have columns forming an orthonormal basis of the space spanned by the vectors $\vecc(YR)$ for $R\in\Rpp$: such $V$ exists because $\rank(Y) = p$. Indeed, $\vecc(YR) = (I_p \otimes Y)\vecc(R)$ and $I_p \otimes Y \in \reals^{np \times p^2}$ then has full rank $p^2$. Since $Y$ is a critical point, $SY = 0$ by~\eqref{eq:firstordercondition}, which implies $(I_p\otimes S)V = 0$.
	
	Let $k'$ denote the dimension of the space spanned by the columns of both $U$ and $V$, and let $W\in\reals^{np\times k'}, W\transpose W = I_{k'}$, be an orthonormal basis for this space. It follows that $M = W\transpose (I_p\otimes S)W$ is positive semidefinite. Indeed, for any $z$, there exist $x, y$ such that $Wz = Ux + Vy$. Hence, $z\transpose M z = x\transpose U\transpose(I_p \otimes S)U x \geq 0$.
	
	Let $\lambda_0 \leq \cdots \leq \lambda_{n-1}$ denote the eigenvalues of $S$, and let $\tilde \lambda_0 \leq \cdots \leq \tilde \lambda_{np-1}$ denote the eigenvalues of $I_p\otimes S$. The latter are simply the eigenvalues of $S$ repeated $p$ times, thus: $\tilde \lambda_i = \lambda_{\floorr{i/p}}$. Let $\mu_0 \leq \cdots \leq \mu_{k'-1}$ denote the eigenvalues of $M$. The Cauchy interlacing theorem
	states that, for all $i$,
	\begin{align}
	\tilde \lambda_i \leq \mu_i \leq \tilde \lambda_{i + np - k'}.
	\label{eq:cauchyinterlace}
	\end{align}
	In particular, since $M\succeq 0$, we have $0 \leq \mu_0 \leq \lambda_{\floorr{(np - k')/p}}$. It remains to determine $k'$.
	
	From Proposition~\ref{prop:submanifold}, recall that $k = \dim \calM = np - m'$. We now investigate how many new dimensions $V$ adds to $U$. All matrices $R\in\Rpp$ admit a unique decomposition as
	\begin{align*}
		R = R_\mathrm{skew} + R_{\ker \calL} + R_{(\ker \calL)^\bot},
	\end{align*}
	where $R_\mathrm{skew}$ is skew-symmetric, $R_{\ker \calL}$ is in the kernel of $\calL_X$~\eqref{eq:calL} and $R_{(\ker \calL)^\bot}$ is in the orthogonal complement of the latter in $\Spp$.
	Recalling the definition of tangent vectors~\eqref{eq:tangentspace}, it is clear that $\dot Y = YR_\mathrm{skew}$ is tangent. Similarly, $\dot Y = YR_{\ker \calL}$ is tangent because of the definition of $\calL_X$~\eqref{eq:calL}. Thus, vectorized versions of these are already in the span of $U$. On the other hand, by definition, $YR_{(\ker \calL)^\bot}$ is not tangent at $Y$ (if it is nonzero). This raises $k'$ (the rank of $W$) by $\dim \, (\ker \calL_X)^\bot = \frac{p(p+1)}{2} - \dim \ker \calL_X$. Since $\dim \ker \calL_X = \dim \calF_X$, we have:
	\begin{align}
	k' & = np - m' + \frac{p(p+1)}{2} - \dim \calF_X = np + \Delta - \dim \calF_X.
	\label{eq:kprime}
	\end{align}
	Thus, $np-k' = \dim \calF_X - \Delta$. Combine with $\lambda_{\floorr{(np - k')/p}} \geq 0$ to conclude. 
\end{proof}
Theorem~\ref{thm:deterministic} follows easily from Corollary~\ref{cor:YoptimalifSpsd} and Theorem~\ref{thm:eigS}.
\begin{remark}\label{rem:highfacedim}
	What does it take for a second-order critical point $Y\in\calM$ to be suboptimal? For local optima, the quote from Burer and Monteiro~\citep[\S3]{burer2005local} in the introduction readily states that $Y$ must have rank $p$, and the face $\calF_{X}$ (with $X = YY\transpose$) must be positive dimensional and such that the cost function $\inner{C}{X}$ is constant over $\calF_X$.
	Here, under Assumption~\ref{assu:M} for $p$, Theorem~\ref{thm:eigS} states that if $Y$ is second-order critical
	and is suboptimal, then $\calF_X$ must have dimension $\Delta + p$ or higher. Since~\eqref{eq:dimF} suggests generic faces at rank $p$ have dimension $\Delta$, this further shows thats suboptimal second-order critical points, if they exist, can only occur if the cost function is constant over a high-dimensional face of $\calC$.
%
%
\end{remark}
%
%
To use Theorem~\ref{thm:eigS} in a particular application, one needs to obtain upper bounds on the dimensions of faces of $\calC$. We follow this path for a number of examples in Section~\ref{sec:applications}.

\section{Near optimality of near second-order critical points}

Under Assumption~\ref{assu:M}, problem~\eqref{eq:P} is an example of smooth optimization over a smooth manifold. This suggests using \emph{Riemannian optimization} to solve it~\citep{AMS08}, as already proposed by Journ\'ee et al.~\citep{journee2010low} in a similar context.
Importantly, known algorithms---in particular, the \emph{Riemannian trust-region method} (RTR)---converge to second-order critical points regardless of initialization~\citep{genrtr}.
We state here a recent computational result to that effect~\citep{boumal2016globalrates}.
\begin{proposition}\label{prop:complexity}
	Under Assumption~\ref{assu:M}, if $\calC$ is compact, RTR initialized with any $Y_0\in\calM$ produces in $\calO(1/\varepsilon_g^2\varepsilon_H^{} + 1/\varepsilon_H^3)$ iterations a point $Y\in\calM$ such that
	\begin{align*}
		g(Y) & \leq g(Y_0), & \|\grad g(Y)\| & \leq \varepsilon_g, & \textrm{ and } & & \Hess g(Y) \succeq -\varepsilon_H \Id,
	\end{align*}
	where $g$~\eqref{eq:g} is the cost function of~\eqref{eq:P}.
\end{proposition}
\begin{proof}
	Apply the main results of~\citep{boumal2016globalrates} using the fact that $g$ has locally Lipschitz continuous gradient and Hessian in $\Rnp$ and $\calM$ is a compact submanifold of $\Rnp$.
\end{proof}

Importantly, only a finite number of iterations of any algorithm can be run in practice, so that only approximate second-order critical points can be computed. Thus, it is of interest to establish whether approximate second-order critical points are also approximately optimal.
As a first step, we give a soft version of Corollary~\ref{cor:YoptimalifSpsd}. We remark that the condition $I_n \in \im\calA^*$ is satisfied in all examples of Section~\ref{sec:applications}.
\begin{lemma} \label{lem:optimgap}
	Let Assumption~\ref{assu:M} hold for some $p$ and assume $\calC$~\eqref{eq:calC} is compact.
	For any $Y$ on the manifold $\calM$, if $\|\grad g(Y)\| \leq \varepsilon_g$ and $S(Y) \succeq -\frac{\varepsilon_H}{2} I_n$, then the optimality gap at $Y$ with respect to~\eqref{eq:SDP} is bounded as
	\begin{align}
		0 \leq 2(g(Y) - f^\star) \leq \varepsilon_H R + \varepsilon_g \sqrt{R},
		\label{eq:lemoptimgap}
	\end{align}
	where $f^\star$ is the optimal value of~\eqref{eq:SDP} and $R = \max_{X\in\calC} \trace(X) < \infty$ measures the size of $\calC$.
		
	If $I_n\in\im(\calA^*)$,
	the right-hand side of~\eqref{eq:lemoptimgap} can be replaced by $\varepsilon_H R$. This holds in particular if all $X\in\calC$ have same trace and $\calC$ has a relative interior point (Slater condition).
\end{lemma}
\begin{proof}
	By assumption on $S(Y) = C - \calA^*(\mu(Y))$~\eqref{eq:S} with $\mu(Y) = G^\dagger \calA(CYY\transpose)$,
	\begin{align*}
		\forall X' \in \calC, \quad -\frac{\varepsilon_H}{2} \trace(X') \leq \innersmall{S(Y)}{X'} & = \innersmall{C}{X'} - \innersmall{\calA^*(\mu(Y))}{X'} \\
		& = \innersmall{C}{X'} - \innersmall{\mu(Y)}{b}.
	\end{align*}
	This holds in particular for $X'$ optimal for~\eqref{eq:SDP}. Thus, we may set $\innersmall{C}{X'} = f^\star$; and certainly, $\trace(X') \leq R$. Furthermore,
	\begin{align*}
		\innersmall{\mu(Y)}{b} = \innersmall{\mu(Y)}{\calA(YY\transpose)} = \innersmall{C-S(Y)}{YY\transpose} = g(Y) - \innersmall{S(Y)Y}{Y}.
	\end{align*}
	Combining the displayed equations and using $\grad g(Y) = 2S(Y)Y$~\eqref{eq:firstordercondition}, we find
	\begin{align}
		0 \leq 2(g(Y) - f^\star) \leq \varepsilon_H R + \innersmall{\grad g(Y)}{Y}.
		\label{eq:optgappartial}
	\end{align}
	In general, we do not assume $I_n \in \im(\calA^*)$ and we get the result by Cauchy--Schwarz on~\eqref{eq:optgappartial} and $\|Y\| = \sqrt{\trace(YY\transpose)} \leq \sqrt{R}$:
	\begin{align*}
		0 \leq 2(g(Y)-f^\star) \leq \varepsilon_H R + \varepsilon_g \sqrt{R}.
	\end{align*}
	But if $I_n \in \im(\calA^*)$, then we show that $Y$ is a normal vector at $Y$, so that it is orthogonal to $\grad g(Y)$. Formally: there exists $\nu \in \Rm$ such that $I_n = \calA^*(\nu)$, and
	\begin{align*}
		\innersmall{\grad g(Y)}{Y} & = \innersmall{\grad g(Y) Y\transpose}{I_n} = \innersmall{\calA(\grad g(Y) Y\transpose)}{\nu} = 0,
	\end{align*}
	since $\grad g(Y) \in \T_Y\calM$~\eqref{eq:tangentspace}. This indeed allows us to simplify~\eqref{eq:optgappartial}.
	
	To conclude, we show that if $\calC$ has a relative interior point $X'$ (that is, $\calA(X') = b$ and $X' \succ 0$) and if $\Trace(X)$ is constant for $X$ in $\calC$, then $I_n \in \im(\calA^*)$. Indeed, $\Snn = \im(\calA^*) \oplus \ker\calA$, so there exist $\nu\in\Rm$ and $M \in \ker \calA$ such that $I_n = \calA^*(\nu) + M$. Thus, for all $X$ in $\calC$,
	\begin{align*}
		0 = \trace(X - X') = \inner{\calA^*(\nu)+M}{X-X'} = \inner{M}{X-X'}.
	\end{align*}
	This implies $M$ is orthogonal to all $X-X'$. These span $\ker\calA$ since $X'$ is interior. Indeed, for any $H\in\ker\calA$, since $X'\succ 0$, there exists $t > 0$ such that $X \triangleq X' + t H \succeq 0$ and $\calA(X) = b$, so that $X \in \calC$.
	Hence, $M \in \ker\calA$ is orthogonal to $\ker \calA$. Consequently, $M = 0$ and $I_n = \calA^*(\nu)$.
\end{proof}
The lemma above involves a condition on the spectrum of $S$. Next, we show this condition is satisfied under an assumption on the spectrum of $\Hess g$ and rank deficiency.
\begin{lemma}\label{lem:fromHesstoS}
	Let Assumption~\ref{assu:M} hold for some $p$.
	If $Y\in\calM$ is column-rank deficient and $\Hess g(Y) \succeq -\varepsilon_H \Id$, then $S(Y) \succeq -\frac{\varepsilon_H}{2} I_n$.
\end{lemma}
\begin{proof}
	By assumption, there exists $z\in\Rp$, $\|z\| = 1$ such that $Yz = 0$. Thus, for any $x\in\Rn$, we can form $\dot Y = xz\transpose$: it is a tangent vector since $Y\dot Y\transpose = 0$~\eqref{eq:tangentspace}, and $\|\dot Y\|^2 = \|x\|^2$. Then, condition~\eqref{eq:secondorderconditionbis} combined with the assumption on $\Hess g(Y)$ tells us
	\begin{align*}
		-\varepsilon_H \|x\|^2 \leq \innersmall{\dot Y}{\Hess g(Y)[\dot Y]} = 2\innersmall{\dot Y}{S\dot Y} = 2\innersmall{xz\transpose z x\transpose}{S} = 2 x\transpose S x.
	\end{align*}
	This holds for all $x\in\Rn$, hence $S \succeq -\frac{\varepsilon_H}{2} I_n$ as required.
\end{proof}
We now combine the two previous lemmas to form a soft optimality statement.
\begin{theorem} \label{thm:approxtolerance}
	Assume $\calC$ is compact and let $R < \infty$ be the maximal trace of any $X$ feasible for~\eqref{eq:SDP}.
	For some $p$, let Assumption~\ref{assu:M} hold for both $p$ and $p+1$.
	For any $Y\in\calM_p$, form $\tilde Y = \left[ Y | 0_{n\times 1} \right]$ in $\calM_{p+1}$.
	The optimality gap at $Y$ is bounded as
	\begin{align}
		0 \leq 2(g(Y)-f^\star) \leq \sqrt{R} \|\grad g(Y)\| - R \lambdamin(\Hess g(\tilde Y)).
		\label{eq:lambdaminbound}
	\end{align}
	If all $X \in \calC$ have the same trace $R$ and there exists a positive definite feasible $X$, then the bound
	\begin{align}
		0 \leq 2(g(Y)-f^\star) \leq - R \lambdamin(\Hess g(\tilde Y))
	\end{align}
	holds.
	If $p>n$, the bounds hold with $\tilde Y = Y$ (and Assumption~\ref{assu:M} only needs to hold for $p$.) 
\end{theorem}
\begin{proof}
	Since $\tilde Y \tilde Y\transpose = YY\transpose$, $S(\tilde Y) = S(Y)$; in particular, we have $g(\tilde Y) = g(Y)$ and $\|\grad g(\tilde Y)\| = \|\grad g(Y)\|$. Since $\tilde Y$ has deficient column rank, apply Lemmas~\ref{lem:optimgap} and~\ref{lem:fromHesstoS}. For $p > n$, there is no need to form $\tilde Y$ as $Y$ itself necessarily has deficient column rank.
\end{proof}
This works well with Proposition~\ref{prop:complexity}. Indeed, equation~\eqref{eq:lambdaminbound} also implies the following:
\begin{align*}
	\lambdamin(\Hess g(\tilde Y)) \leq - \frac{2(g(Y)-f^\star) - \sqrt{R}\|\grad g(Y)\|}{R}.
\end{align*}
That is,
an approximate critical point $Y$ in $\calM_p$ which is far from optimal (for~\eqref{eq:SDP}) maps to a comfortably-escapable approximate saddle point $\tilde Y$ in $\calM_{p+1}$. This can be helpful for the development of optimization algorithms.

For $p=n+1$, the bound in Theorem~\ref{thm:approxtolerance} can be controlled a priori: approximate second-order critical points are approximately optimal, for any $C$.\footnote{With $p=n+1$, problem~\eqref{eq:P} is no longer lower dimensional than~\eqref{eq:SDP}, but retains the advantage of not involving a positive semidefiniteness constraint.}
\begin{corollary}\label{cor:approxsocpapproxoptimalpnplusone}
	Assume $\calC$ is compact. Let Assumption~\ref{assu:M} hold for $p = n+1$.
	If $Y\in\calM_{n+1}$ satisfies both $\|\grad g(Y)\| \leq \varepsilon_g$ and $\Hess g(Y) \succeq -\varepsilon_H \Id$, then $Y$ is approximately optimal in the sense that (with $R = \max_{X\in\calC} \trace(X)$):
	\begin{align*}
		0 \leq 2(g(Y)-f^\star) \leq \varepsilon_g\sqrt{R} + \varepsilon_H R.
	\end{align*}
	Under the same condition as in Theorem~\ref{thm:approxtolerance}, the bound holds with right-hand side $\varepsilon_H R$ instead.
\end{corollary}
Theorem~\ref{thm:stableinformal} is an informal statement of this corollary.

\section{Applications} \label{sec:applications}


In all applications below, Assumption~\ref{assu:M}a holds for all $p$ such that the search space is non-empty. For each one, we deduce the consequences of Theorems~\ref{thm:mastersmallp} and~\ref{thm:deterministic}. For the latter, the key part is to investigate the facial structure of the SDP. As everywhere else in the paper, $\|x\|$ denotes the 2-norm of vector $x$ and $\|X\|$ denotes the Frobenius norm of matrix $X$.

\subsection{Generalized eigenvalue SDP}

The generalized symmetric eigenvalue problem admits a well-known extremal formulation:
\begin{align}
	\min_{x \in \Rn} x\transpose C x \quad \st \quad x\transpose B x = 1,
	\tag{EIG}
	\label{eq:geneig}
\end{align}
where $C, B$ are symmetric of size $n \geq 2$. The usual relaxation by lifting introduces $X = xx\transpose$ and discards the constraint $\rank(X) = 1$ to obtain this SDP (which is also the Lagrangian dual of the dual of~\eqref{eq:geneig}):
\begin{align}
	\min_{X\in\Snn} \inner{C}{X}  \quad \st \quad \inner{B}{X} = 1, \ X \succeq 0.
	\tag{EIG-SDP}
	\label{eq:SDPgeneig}
\end{align}
Let $\calC$ denote the search space of~\eqref{eq:SDPgeneig}. It is non-empty and compact if and only if $B \succ 0$, which we now assume.
A direct application of~\eqref{eq:patakirank} guarantees all extreme points of $\calC$ have rank 1, so that it always admits a solution of rank 1: the SDP relaxation is always tight, which is well known. 
Under our assumption, $B$ admits a Cholesky factorization as $B = R\transpose R$ with $R \in \Rnn$ invertible. The corresponding Burer--Monteiro formulation at rank $p$ reads:
\begin{align}
	\min_{Y \in \Rnp} \inner{CY}{Y} \quad \st \quad \|RY\|^2 = 1.
	\tag{EIG-BM}
	\label{eq:BMgeneig}
\end{align}
Let $\calM$ denote its search space. Assumption~\ref{assu:M}a holds for any $p \geq 1$ with $m' = 1$. Indeed, for all $Y \in \calM$, $\{ BY \}$ spans a subspace of dimension 1, since $BY = R\transpose R Y$, $RY \neq 0$ and $R\transpose$ is invertible. Thus, Theorem~\ref{thm:mastersmallp} readily states that for $p \geq 2$, for almost all $C$, all second-order critical points of~\eqref{eq:BMgeneig} are optimal.

We can do better. The facial structure of $\calC$ is easily described. Recalling~\eqref{eq:dimF}, for all $X = YY\transpose \in \calC$ we have $\dim \calF_X = \frac{p(p+1)}{2} - 1$, since $Y\transpose B Y \neq 0$. Hence, by Theorem~\ref{thm:deterministic}, for any value of $p \geq 1$, all second-order critical points of~\eqref{eq:BMgeneig} are optimal (for any $C$). In particular, for $p = 1$~\eqref{eq:geneig} and~\eqref{eq:BMgeneig} coincide and we get:
\begin{corollary}
	All second-order critical points of~\eqref{eq:geneig} are optimal.
\end{corollary}
This is a well-known  fact, though usually proven by direct inspection of necessary optimality conditions. 


\subsection{Trust-region subproblem SDP} \label{sec:trs}
The trust-region subproblem consists in minimizing a quadratic on a sphere, with $n \geq 2$:
\begin{align}
	\min_{x \in \Rn} x\transpose A x + 2b\transpose x + c \quad \st \quad \|x\|^2 = 1.
	\tag{TRS}
	\label{eq:TRS}
\end{align}
It is not difficult to produce $(A, b, c)$ such that~\eqref{eq:TRS} admits suboptimal second-order critical points.
The usual lifting here introduces
\begin{align*}
	X & = \begin{pmatrix}
	x \\ 1
	\end{pmatrix} \begin{pmatrix}
	x\transpose & 1
	\end{pmatrix} = \begin{pmatrix}
	xx\transpose & x \\ x\transpose & 1
	\end{pmatrix}, & \textrm{ and } & & C & = \begin{pmatrix}
	A & b \\ b\transpose & c
	\end{pmatrix}.
\end{align*}
The quadratic cost and constraint are linear in $X$, yielding this SDP relaxation:
\begin{align}
	\min_{X\in\Snn} \inner{C}{X}  \quad \st \quad  \trace(X_{1:n, 1:n}) = 1, \ X_{n+1, n+1} = 1, \ X \succeq 0.
	\tag{TRS-SDP}
	\label{eq:SDPTRS}
\end{align}
Let $\calC$ denote the search space of~\eqref{eq:SDPTRS}. It is non-empty and compact.
Here too, a direct application of~\eqref{eq:patakirank} guarantees the SDP relaxation is always tight (it always admits a solution of rank 1), which is a well-known fact related to the S-lemma~\citep{polik2007slemma}.
The Burer--Monteiro relaxation at rank $p$ reads:
\begin{align}
	\min_{Y_1 \in \Rnp, y_2 \in \Rp} \inner{CY}{Y}  \quad \st \quad  \|Y_1\|^2 = 1, \ \|y_2\|^2 = 1, \quad \textrm{ with } Y = \begin{pmatrix}
	Y_1 \\ y_2\transpose
	\end{pmatrix}.
	\tag{TRS-BM}
	\label{eq:BMTRS}
\end{align}
Let $\calM$ denote its search space. After verifying Assumption~\ref{assu:M} holds (see below), application of Theorem~\ref{thm:mastersmallp} guarantees that for $p \geq 2$ and for \emph{almost} all $(A, b, c)$, second-order critical points of~\eqref{eq:BMTRS} are optimal. We can further strengthen this result by looking at the faces of $\calC$, as we do now.
\begin{lemma}
	Assumption~\ref{assu:M}a holds for any $p \geq 1$ with $m' = 2$. Furthermore, for $X \in \calC$ of rank $p$, 
	\begin{align*}
		\dim \calF_X & = \begin{cases}
		0 & \textrm{ if } p = 1, \\
		\frac{p(p+1)}{2} - 2 & \textrm{ if } p \geq 2.
		\end{cases}
	\end{align*}
\end{lemma}
\begin{proof}
The constraints of~\eqref{eq:SDP} are defined by
\begin{align*}
	A_1 & = \begin{pmatrix}
	I_n & 0_{n\times 1} \\ 0_{1 \times n} & 0
	\end{pmatrix}, & b_1 & = 1, & A_2 & = \begin{pmatrix}
	0_{n\times n} & 0_{n \times 1} \\ 0_{1 \times n} & 1
	\end{pmatrix}, & b_2 & = 1.
\end{align*}
For $Y \in \calM$, we have
\begin{align*}
	A_1 Y & = \begin{pmatrix}
	Y_1 \\ 0_{1 \times p}
	\end{pmatrix}, & A_2 Y & = \begin{pmatrix}
	0_{n\times p} \\ y_2\transpose
	\end{pmatrix}.
\end{align*}
These are nonzero and always linearly independent, so that $\dim \spann \{A_1Y, A_2Y\} = 2$ for all $Y\in\calM$, which confirms Assumption~\ref{assu:M}a holds with $m' = 2$.

The facial structure of $\calC$ is simple as well. Let $X \in \calC$ have rank $p$ and consider $Y\in\calM$ such that $X = YY\transpose$. To use~\eqref{eq:dimF}, note that:
\begin{align*}
	Y\transpose A_1 Y & = Y_1\transpose Y_1^{}, & Y\transpose A_2 Y & = y_2^{} y_2\transpose.
\end{align*}
These are nonzero. For $p = 1$, they are scalars: they span a subspace of dimension 1. Then, $\dim \calF_X = 1 - 1 = 0$. For $p > 1$, we argue they are linearly independent. Indeed, if they are not, there exists $\alpha \neq 0$ such that $Y_1\transpose Y_1^{} = \alpha \cdot y_2^{} y_2\transpose$. If so, $Y_1$ must have rank 1 with row space spanned by $y_2$, so that $Y_1 = z y_2\transpose$ for some $z \in \Rn$, and $\|z\| = 1$. As a result, $Y$ itself has rank 1, which is a contradiction. Thus, $\dim \calF_X = \frac{p(p+1)}{2} - 2$, as announced.
\end{proof}
Combining the latter with Theorem~\ref{thm:deterministic} yields the following new result, which holds for \emph{all} $(A, b, c)$. Notice that for $p = 1$, the theorem correctly allows second-order critical points to be suboptimal in general.
\begin{corollary}
	For $p \geq 2$, all second-order critical points of~\eqref{eq:BMTRS} are globally optimal.
\end{corollary}
A second-order critical point $Y$ of~\eqref{eq:BMTRS} with $p = 2$ is thus always optimal. If $Y$ has rank 1, it is straightforward to extract a solution of~\eqref{eq:TRS} from it.
If $Y$ has rank 2,\footnote{This can happen, notably if $(A, b, c)$ forms a so-called \emph{hard case} TRS (details omitted.) This observation shows that it is indeed necessary to exclude some non-trivial matrices $C$ in Lemma~\ref{lem:criticalptsrankdeficient}.}
it maps to a face of dimension~1. The endpoints of that face have rank 1 and are also optimal.
The following lemma shows these can be computed easily from $Y$ by solving two scalar equations. 
\begin{lemma}
	Let $Y \in \calM$ be a second-order critical point of~\eqref{eq:BMTRS} with $p = 2$, and let $z \in \reals^2$ satisfy $\|Y_1z\|^2 = 1$ and $y_2\transpose z = 1$. Then, $Y_1 z$ is a global optimum of~\eqref{eq:TRS}.
\end{lemma}
\begin{proof}
	If $\rank(Y) = 1$, then $Y_1 = xy_2^T$ for some $x \in \Rn$, and $\|Y_1\| = 1, \|y_2\| = 1$ ensure $\|x\| = 1$. Solutions to $y_2^T z = 0$ are of the form $z = y_2 + u$, where $y_2^T u = 0$. For any such $z$, $Y_1 z = x$, which is indeed optimal for~\eqref{eq:TRS} since $Y$ is globally optimal for~\eqref{eq:BMTRS} and $x$ attains the same cost for the restricted problem~\eqref{eq:TRS}.

	Now assume $\rank(Y) = 2$. By~\eqref{eq:calF}, the one-dimensional face $\calF_{YY\transpose}$ contains all matrices of the form $Y(I_2 - M)Y\transpose$ such that $I_2 - M \succeq 0$ and $\inner{I_2-M}{Y_1\transpose Y_1^{}} = 0$, $\inner{I_2-M}{y_2^{} y_2\transpose} = 0$. This face has two extreme points of rank 1, for which $I_2 - M$ is a positive semidefinite matrix of rank 1, so that $I_2 - M = zz\transpose$ for some $z \in \reals^2$. Given that $Y$ is feasible, the conditions on $z$ are $\|Y_1z\|^2 = 1$ and $y_2\transpose z = \pm 1$. These equations define an ellipse in $\reals^2$ and two parallel lines, totaling four intersections $\pm z, \pm z'$ which can be computed explicitly. Fixing $y_2\transpose z = +1$ allows to identify the two extreme points of the face. Since the cost function is constant along that face, either extreme point yields a global optimum in the same way as above.
\end{proof}

\subsection{Optimization over several spheres}

The trust-region subproblem generalizes to optimization of a quadratic function over $k$ spheres, possibly in different dimensions $n_1, \ldots, n_k \geq 2$:
\begin{align}
	& \min_{x_i \in \reals^{n_i}, i = 1 \ldots k} x\transpose C x \quad \st \quad \|x_1\| = \cdots = \|x_k\| = 1,
	\tag{Spheres} \label{eq:manyspheres} \\
	& \textrm{with } x\transpose = \begin{pmatrix}
	x_1\transpose & \cdots & x_k\transpose & 1
	\end{pmatrix}. \nonumber	
\end{align}
The variable $x$ is in $\reals^{n+1}$, with $n = n_1 + \cdots + n_k$. Since the last entry of $x$ is 1, this indeed covers all possible quadratic functions of $x_1, \ldots, x_k$.
The SDP relaxation by lifting reads:
\begin{align}
	 \min_{X \in \reals^{(n+1) \times (n+1)}} \inner{C}{X} \quad \st \quad & \trace(X_{11}) = \cdots = \trace(X_{kk}) = 1, \nonumber\\ & X_{n+1, n+1} = 1, X \succeq 0,
	\tag{Spheres-SDP}
	\label{eq:manyspheresSDP}
\end{align}
where $X_{ij}$ denotes the block of size $n_i \times n_j$ of matrix $X$, in the obvious way. This SDP has a non-empty compact search space and $k+1$ independent constraints, so that by~\eqref{eq:patakirank} it always admits a solution of rank at most $p^* = \frac{\sqrt{8k+9}-1}{2}$. The Burer--Monteiro relaxation at rank $p$ reads:
\begin{align}
	& \min_{Y \in \reals^{(n+1)\times p}} \inner{CY}{Y} \quad \st \quad \|Y_1\| = \cdots = \|Y_k\| = 1, \|y\| = 1,
	\tag{Spheres-BM}
	\label{eq:manyspheresBM} \\
	& \textrm{with } Y\transpose = \begin{pmatrix}
	Y_1\transpose & \cdots & Y_k\transpose & y
	\end{pmatrix}, \nonumber	
\end{align}
where $Y_i \in \reals^{n_i \times p}$ and $y \in \reals^p$. It is easily checked that Assumption~\ref{assu:M}a holds for all $p \geq 1$. Thus, Theorem~\ref{thm:mastersmallp} gives this result:
\begin{corollary}
	For $p > \frac{\sqrt{8k+9}-1}{2}$ and for almost all $C$, all second-order critical points of~\eqref{eq:manyspheresBM} are optimal and map to optima of~\eqref{eq:manyspheresSDP}.
\end{corollary}
To apply Theorem~\ref{thm:deterministic}, we first investigate the facial structure of the SDP.
\begin{lemma}
	Let $Y$ be feasible for~\eqref{eq:manyspheresBM} and have full rank $p$. The dimension of the face of the search space of~\eqref{eq:manyspheresSDP} at $YY\transpose$ obeys:
	\begin{align*}
		\dim \calF_{YY\transpose} & \leq \frac{p(p+1)}{2} - 2
	\end{align*}
	if $p \geq 2$, and $\dim \calF_{YY\transpose} = 0$ if $p = 1$.
\end{lemma}
\begin{proof}
	Following~\eqref{eq:dimF},
	\begin{align*}
		\dim \calF_{YY\transpose} & = \frac{p(p+1)}{2} - \dim \spann \left( Y_1\transpose Y_1^{}, \ldots, Y_k\transpose Y_k^{}, yy\transpose \right).
	\end{align*}
	Since $Y$ is feasible, each defining element of the span is nonzero, so that the dimension is at least 1. If $p = 1$, these elements are scalars: they span $\reals$. Now consider $p \geq 2$ and assume for contradiction that the span has dimension one. Then, all defining elements are equal up to scaling. In other words: $Y_i\transpose Y_i^{} = \alpha_i \cdot yy\transpose$ for some nonzero $\alpha_i$. If so, $Y_i$ has rank 1 and there exists $z_i \in \reals^{n_i}$ such that $Y_i = z_i y\transpose$. In turn, this implies $Y$ has rank 1, which is a contradiction. Thus, the span has dimension at least two.
\end{proof}
\begin{corollary}
	For $p \geq \max(2, k)$, all second-order critical points of~\eqref{eq:manyspheresBM} are optimal and map to optima of~\eqref{eq:manyspheresSDP} (for any $C$).
\end{corollary}
For $k = 1$, this recovers the main result about the trust-region subproblem.
If the cost function in~\eqref{eq:manyspheres} is a homogeneous quadratic, then it can be written as
\begin{align}
	& \min_{x_i \in \reals^{n_i}, i = 1 \ldots k} x\transpose C x \quad \st \quad \|x_1\| = \cdots = \|x_k\| = 1,
	\tag{SpheresH} \label{eq:manysphereshomog} \\
	& \textrm{with } x\transpose = \begin{pmatrix}
	x_1\transpose & \cdots & x_k\transpose
	\end{pmatrix}. \nonumber	
\end{align}
The corresponding relaxation and Burer--Monteiro formulations read:
\begin{align}
	& \min_{X \in \reals^{n \times n}} \inner{C}{X} \quad \st \quad \trace(X_{11}) = \cdots = \trace(X_{kk}) = 1, X \succeq 0,
	\tag{SpheresH-SDP}
	\label{eq:manysphereshomogSDP}
\end{align}
and:
\begin{align}
	& \min_{Y \in \reals^{n\times p}} \inner{CY}{Y} \quad \st \quad \|Y_1\| = \cdots = \|Y_k\| = 1,
	\tag{SpheresH-BM}
	\label{eq:manysphereshomogBM} \\
	& \textrm{with } Y\transpose = \begin{pmatrix}
	Y_1\transpose & \cdots & Y_k\transpose
	\end{pmatrix}. \nonumber	
\end{align}
Assumption~\ref{assu:M}a holds for all $p \geq 1$ with $m' = k$. A similar analysis of the facial structure yields the following corollary of Theorem~\ref{thm:deterministic}.
\begin{corollary} 
	For almost all $C$,  provided $p > \frac{\sqrt{8k+1}-1}{2}$, all second-order critical points of~\eqref{eq:manysphereshomogBM} are optimal and map to optima of~\eqref{eq:manysphereshomogSDP}. If $p \geq k$, the result holds for all $C$.
\end{corollary}
For $k = 1$, this recovers the results of~\eqref{eq:geneig} with $B = I_n$.

\subsection{Max-Cut and Orthogonal-Cut SDP}

Let $n = qd$ for some integers $q, d$.
Consider the semidefinite program
\begin{align}
	\min_{X\in\Snn} \inner{C}{X}  \quad \st \quad  \sbd{X} = I_n, \ X \succeq 0,
	\tag{OrthoCut}
	\label{eq:orthocut}
\end{align}
where $\sbdop \colon \Snn \to \Snn$ preserves the diagonal blocks of size $d \times d$ and zeros out all other blocks. Specifically, with $X_{ij}$ denoting the $(i,j)$th block of size $d\times d$ in matrix $X$,
\begin{align*}
	\sbd{X}_{ij} & = \begin{cases}
	X_{ii} & \textrm{ if } i = j,\\
	0_{d\times d} & \textrm{ otherwise.}
	\end{cases}
\end{align*}
For example, with $d = 1$, the constraint $\sbd{X} = I_n$ is equivalent to $\diag(X) = \mathbf{1}$ and this SDP is the Max-Cut SDP~\citep{goemans1995maxcut}. For general $d$, diagonal blocks of $X$ of size $d \times d$ are constrained to be identity matrices: this SDP is known as Orthogonal-Cut~\citep{bandeira2013approximating,boumal2015staircase}. Among other uses, it appears as a relaxation of synchronization on $\mathbb{Z}_2 = \{\pm1\}$~\citep{bandeira2016lowrankmaxcut,mei2017solvingSDPs,abbe2014exact} and synchronization of rotations~\citep{rosen2016certifiably,eriksson2017rotationduality}, with applications in stochastic block modeling (community detection) and SLAM (simultaneous localization and mapping for robotics).

The Stiefel manifold $\St(p, d)$ is the set of matrices of size $p \times d$ with orthonormal columns. The Burer--Monteiro formulation of~\eqref{eq:orthocut} is an optimization problem over $q$ copies of $\St(p, d)$:
\begin{align}
	\min_{Y_1, \ldots, Y_q \in \Rpd} \inner{CY}{Y} \quad \st \quad Y_k\transpose Y_k^{} = I_d \ \forall k, \  Y\transpose = \begin{bmatrix}
	Y_1 & \cdots & Y_q
	\end{bmatrix}.
	\tag{OrthoCut-BM}
	\label{eq:orthocutbm}
\end{align}
For $d = 1$, this problem captures one side of the Grothendieck inequality~\citep[eq.~(1.1)]{khot2012grothendieck}.
Assumption~\ref{assu:M}a holds for all $p \geq d$ with $m' = q\frac{d(d+1)}{2}$ (which is the number of constraints). 
Theorem~\ref{thm:mastersmallp} applies as follows.
\begin{corollary}
	If $p > \frac{\sqrt{1 + 4n(d+1)} - 1}{2}$, for almost all $C$, any second-order critical point $Y$ of~\eqref{eq:orthocutbm} is a global optimum, and $X = YY\transpose$ is globally optimal for~\eqref{eq:orthocut}.
\end{corollary}

In order to apply Theorem~\ref{thm:deterministic}, we must investigate the facial structure of
\begin{align*}
	\calC & = \{ X \in \Snn : \sbd{X} = I_n, X \succeq 0 \}.
\end{align*}
The following result generalizes a result in~\citep[Thm.\,3.1(i)]{laurent1996facial} to $d\geq 1$.
\begin{theorem}\label{thm:dimFbounds}
	If $X\in\calC$ has rank $p$, then the face $\calF_X$~\eqref{eq:calF} has dimension bounded as:
	\begin{align}
		\frac{p(p+1)}{2} - n\frac{d+1}{2} \ \leq \  \dim \calF_X \ \leq \  \frac{p(p+1)}{2} - p\frac{d+1}{2}.
	\end{align}
	If $p$ is an integer multiple of $d$, the upper bound is attained for some $X$.
\end{theorem}
The proof is in Appendix~\ref{apdx:facesorthocut}. Combining this with Theorem~\ref{thm:deterministic} yields the following result.
\begin{corollary}
	If $p > \frac{d+1}{d+3}n$, any second-order critical point $Y$ for~\eqref{eq:orthocutbm} is globally optimal, and $X = YY\transpose$ is globally optimal for~\eqref{eq:orthocut}. In particular, for Max-Cut SDP ($d = 1$), the requirement is $p > \frac{n}{2}$.
\end{corollary}
\begin{proof}
	If $Y$ is rank deficient, use Proposition~\ref{prop:rankdeficientY}. Otherwise,
	since $\rank(X) = p$, Theorem~\ref{thm:dimFbounds} gives $\dim \calF_X \leq \frac{p(p+1)}{2} - p\frac{d+1}{2}$ and Theorem~\ref{thm:deterministic} gives optimality if
	\begin{align*}
	\dim \calF_X < \frac{p(p+1)}{2} - n \frac{d+1}{2} + p.
	\end{align*}
	This is the case provided $(n-p)(d+1) < 2p$, that is, if $p > \frac{d+1}{d+3}n$.
\end{proof}

\section{Discussion of the assumptions} \label{sec:discussionassumptions}

We now discuss the assumptions that appear in the main theorems.

The starting point of this investigation is the hope to solve~\eqref{eq:SDP} by solving~\eqref{eq:P} instead.
For smooth, non-convex optimization problems, even verifying local optimality is usually hard~\citep{murty1987npcomplete}. Thus, we wish to restrict our attention to efficiently computable points, such as points which satisfy first- and second-order Karush--Kuhn--Tucker (KKT) conditions for~\eqref{eq:P}---see~\cite[\S2.2]{sdplr} and \cite[\S3]{ruszczynski2006nonlinear}. This only helps if global optima satisfy the latter, that is, if KKT conditions are necessary for optimality.

A global optimum $Y$ necessarily satisfies KKT conditions if \emph{constraint qualifications} (CQs) hold at $Y$~\citep{ruszczynski2006nonlinear}. The standard CQs for equality constrained programs are Robinson's conditions or metric regularity (they are here equivalent). They read as follows:
\begin{align}
	\textrm{CQs hold at } Y\in\calM \textrm{ if } A_1Y, \ldots, A_mY \textrm{ are linearly independent in } \Rnp.
	\tag{CQ}
	\label{eq:CQ}
\end{align}
Considering all cost matrices $C$, global optima could, a priori, be anywhere in $\calM$.
Thus, we require CQs to hold at all $Y$ in $\calM$ rather than only at the (unknown) global optima. This leads to Assumption~\ref{assu:M}a. Adding redundant constraints (for example, duplicating $\inner{A_1}{X} = b_1$) would break the CQs, but does not change the optimization problem. This is allowed by Assumption~\ref{assu:M}b. 

In general, \eqref{eq:SDP} may not have an optimal solution. One convenient way to guarantee that it does is to require $\calC$ to be compact, which is why this assumption appears in Theorem~\ref{thm:stableinformal} to bound optimality gaps for approximate second-order critical points. When $\calC$ is compact, one furthermore gets the guarantee that at least one of the global optima is an extreme point of $\calC$, which leads to the guarantee that at least one of the global optima has rank $p$ bounded as $\frac{p(p+1)}{2} \leq m'$~\eqref{eq:patakirank}. The other way around, it is possible to pick the cost matrix $C$ such that the unique solution to~\eqref{eq:SDP} is an extreme point of maximal rank, which can be as large as allowed by~\eqref{eq:patakirank}. This justifies why, in Theorem~\ref{thm:mastersmallp}, the bound on $p$ is essentially optimal. The compactness assumption could conceivably be relaxed, provided candidate global optima remain bounded. This could plausibly come about by restricting attention to positive definite cost matrices $C$.

One restriction in particular in Theorem~\ref{thm:mastersmallp} merits further investigation: the exclusion of a zero-measure set of cost matrices (``bad $C$'').
From the trust-region subproblem example in Section~\ref{sec:trs}, we know that it is necessary (in general) to allow the exclusion of a zero-measure set of cost matrices in Lemma~\ref{lem:criticalptsrankdeficient}. Yet, in that same example, the excluded cost matrices do not give rise to suboptimal second-order critical points (as we proved through a different argument involving Theorem~\ref{thm:deterministic}.) Thus, it remains unclear whether or not a zero-measure set of cost matrices must be excluded in Theorem~\ref{thm:mastersmallp}.
Resolving this question is key to gain deeper understanding of the relationship between~\eqref{eq:SDP} and~\eqref{eq:P}. 


Finally, we connect the notion of smooth SDP used in this paper to the more standard notion of non-degeneracy in SDPs as defined in~\citep[Def.~5]{alizadeh1997complementarity}. Informally: for linearly independent $A_i$, non-degeneracy at all points is equivalent to smoothness. The proof is in Appendix~\ref{apdx:equivalencenondegeneracysmoothness}.
\begin{definition}
	$X$ is \emph{primal non-degenerate} for~\eqref{eq:SDP} if it is feasible and $\T_X + \ker\calA = \Snn$, where $\T_X$ is the tangent space at $X$ to the manifold of symmetric matrices of rank $r$ embedded in $\Snn$, where $r = \rank(X)$.
\end{definition}
\begin{proposition}\label{prop:equivalencenondegeneracysmoothness}
	Let $A_1, \ldots, A_m$ defining $\calA$ be linearly independent. Then, Assumption~\ref{assu:M}a holds for all $p$ such that $\calM_p$ is non-empty if and only if all $X \in \calC$ are primal non-degenerate.
\end{proposition}

\section{Conclusions and perspectives}

We have shown how, under Assumption~\ref{assu:M} and extra conditions (on $p$, compactness, and the cost matrix), the Burer--Monteiro factorization approach to solving~\eqref{eq:SDP} is ``safe'', despite non-convexity. For future research, it is of interest to determine if
the proposed assumptions can be relaxed.
Furthermore, it is important for practical purposes to determine whether approximate second-order critical points are approximately optimal for values of $p$ well below $n$ (an example of this for a specific context is given in~\citep{bandeira2016lowrankmaxcut}).
One possible way forward is a smoothed analysis of the type developed recently in~\citep{bhojanapalli2018smoothedsdp,pumir2018smoothedsmoothsdp}, though these early works leave plenty of room for improvement.

        

\appendices

\section{{\protect{Consequences and properties of Assumption~\ref{assu:M}}}}\label{apdx:assuforallp}

\begin{proof}[Proof of Proposition~\ref{prop:submanifold}]
	The set $\calM$ is defined as the zero level set of $\Phi \colon \Rnp \to \Rm$ where $\Phi(Y) = \calA(YY\transpose) - b$.
	The differential of $\Phi$ at $Y$, $\D\Phi(Y)$, has rank equal to the dimension of the space spanned by $\{ A_1Y, \ldots, A_mY \}$.
	%
	Under Assumption~\ref{assu:M}a, $\D\Phi(Y)$ has full rank $m$ on $\calM$ and the result follows from~\citep[Corollary~5.14]{lee2012smoothmanifolds}. 
	Under Assumption~\ref{assu:M}b, $\D\Phi(Y)$ has constant rank $m'$ in a neighborhood of $\calM$ and the result follows from~\citep[Theorem~5.12]{lee2012smoothmanifolds}. 
	%
\end{proof}

\begin{proof}[Proof of Proposition~\ref{prop:assuforallp}]

First, let Assumption~\ref{assu:M}a hold for some $p$, and consider $p' < p$ such that $\calM_{p'}$ is non-empty. For any $Y' \in \calM_{p'}$, form $Y = \left[ Y' | 0_{n \times (p-p')} \right] \in \Rnp$. Clearly, $Y$ is in $\calM_p$, so that $$m = \dim \spann\{ A_1 Y, \ldots, A_m Y \} = \dim \spann\{ A_1 Y', \ldots, A_m Y'\},$$ as desired. For $p = n$, we now consider the case $p' > n$. Let $Y' \in \calM_{p'}$ and consider its full SVD, $Y' = U\Sigma V\transpose$, with $\Sigma\in\reals^{n\times p'}$. Then, $Y'V$ is in $\calM_{p'}$ as well. Since the last $p'-n$ columns of $\Sigma$ are zero, we have $Y'V = U\Sigma = \left[ Y | 0_{n \times (p'-n)} \right]$ with $Y \in \calM_n$. Thus, as desired,
\begin{align*}
\dim\spann\{ A_1Y', \ldots, A_mY' \} & = \dim\spann\{ A_1Y'V, \ldots, A_mY'V \} \\ & = \dim\spann\{ A_1Y, \ldots, A_mY \} \\ & = m.
\end{align*}

Second, let Assumption~\ref{assu:M}b hold for some $p$, and consider $p' < p$ such that $\calM_{p'}$ is non-empty. For any $Y' \in \calM_{p'}$, form $Y = \left[ Y' | 0_{n \times (p-p')} \right] \in \calM_p$. By assumption, there exists an open ball $B_Y$ in $\Rnp$ of radius $\varepsilon = \varepsilon(Y) > 0$ centered at $Y$ such that $$\dim\spann\{A_1\tilde Y, \ldots, A_m\tilde Y\} = m'$$ for all $\tilde Y \in B_Y$. Let $B_{Y'}$ be the open ball in $\reals^{n\times p'}$ of radius $\varepsilon(Y)$ and center $Y'$. For any $\tilde Y' \in B_{Y'}$, form $\tilde Y = \left[ \tilde Y' | 0_{n \times (p-p')} \right]$. Since $\|\tilde Y - Y\| = \|\tilde Y' - Y'\| \leq \varepsilon$, we have $\tilde Y \in B_Y$, so that
\begin{align*}
m' = \dim\spann\{A_1\tilde Y, \ldots, A_m\tilde Y\} = \dim\spann\{A_1\tilde Y', \ldots, A_m\tilde Y'\}.
\end{align*}
Thus, Assumption~\ref{assu:M}b holds with the open neighborhood of $\calM_{p'}$ consisting of the union of all balls $B_{Y'}$ for $Y' \in \calM_{p'}$ as described above. 
\end{proof}

\section{The facial structure of $\calC$} \label{apdx:faces}

\begin{proof}[Proof of Proposition~\ref{prop:faces}]
	The construction follows~\citep{pataki1998rank} and applies for any linear equality constraints. 
	We first show that if $X'$ is of the form in~\eqref{eq:calF}, then it must be in $\calF_X$. This is clear if $X' = X$. Otherwise, pick $t > 0$ such that $I_p - tA \succeq 0$. Then, $X'$ and $X'' = Y(I_p - tA)Y\transpose$ define a closed line segment in $\calC$ whose relative interior contains $X$. By Definition~\ref{def:face}, this implies $X'$ (and $X''$) are in $\calF_X$.
	
	The other way around, we now show that any point in $\calF_X$ must be of the form of $X'$ in~\eqref{eq:calF}. Let $W\in\Snn$ be such that $X' = X + W$. Since $X$ is in the relative interior of $\calF_X$ which is convex, there exists $t > 0$ such that $X - tW \in \calF_X$.
	%
	Let $Y_\perp \in \reals^{n \times (n-p)}$ be such that $M = \begin{bmatrix}
	Y & Y_\perp
	\end{bmatrix}$ is invertible. We can express $X = YY\transpose$ and $W$ as
	\begin{align*}
	X & = M \begin{bmatrix}
	I_p & 0 \\ 0 & 0
	\end{bmatrix} M\transpose & & \textrm{ and } & W & = M \begin{bmatrix}
	A & B \\ B\transpose & C
	\end{bmatrix} M\transpose.
	\end{align*}
	Then, explicitly, these two matrices must belong to $\calC$:
	\begin{align*}
	X + W & = M \begin{bmatrix}
	I_p + A & B \\ B\transpose & C
	\end{bmatrix} M\transpose, & & \textrm{ and } & X - tW & = M \begin{bmatrix}
	I_p - tA & -tB \\ -tB\transpose & -tC
	\end{bmatrix} M\transpose.
	\end{align*}
	In particular, they must both be positive semidefinite, which implies $C \succeq 0$ and $-tC\succeq 0$, so that $C = 0$. By Schur's complement, it follows that $B = 0$. Thus, $W = YAY\transpose$ for some $A \in \Spp$ such that $I_p + A \succeq 0$. Furthermore, $\calA(X') = \calA({X+W}) = b$, so that $\calA({W}) = 0$. The latter is equivalent to $\calL_X(A) = 0$ using~\eqref{eq:calL}.
	%
	%
	%
%
\end{proof}

\section{Faces of the Ortho-Cut SDP}\label{apdx:facesorthocut}

\begin{proof}[Proof of Theorem~\ref{thm:dimFbounds}]
	Consider the definition of $\calL_X$~\eqref{eq:calL} and inequality~\eqref{eq:dimF}: the latter covers the lower bound and shows we need $\rank \calL_X \geq p(d+1)/2$ for the upper bound, that is, we need to show the condition $\calL_X(A) = 0$ imposes at least $p(d+1)/2$ linearly independent constraints on $A\in\Spp$.
	
	Let $Y\in\calM_p$ be such that $X = YY\transpose$, and let $y_1, \ldots, y_n \in \Rp$ denote the rows of $Y$, transposed. 
	Greedily select $p$ linearly independent rows of $Y$, in order, such that row $i$ is picked iff it is linearly independent from rows $y_1$ to $y_{i-1}$. This is always possible since $Y$ has rank $p$. Write $t = \{ t_1 < \cdots < t_p \}$ to denote the indices of selected rows. Write $s_k = \{ ((k-1)d+1), \ldots, kd \}$ to denote the indices of rows in slice $Y_k\transpose$, and let $c_k = s_k \cap t$ be the indices of selected rows in that slice.
	
	We make use of the following fact~\citep[Lem.\,2.1]{laurent1996facial}: for $x_1,\ldots,x_p\in\Rp$ linearly independent, the $p(p+1)/2$ symmetric matrices $x_i^{}x_j\transpose + x_j^{}x_i\transpose$ form a basis of $\Spp$. Defining $E_{ij} = y_i^{}y_j\transpose + y_j^{}y_i\transpose = E_{ji}$, this means $\calE = \{ E_{t_{\ell^{}},t_{\ell'}} : \ell,\ell' = 1\ldots p \}$ forms a basis of $\Spp$ ($\calE$ is a set, so that $E_{ij}$ and $E_{ji}$ contribute only one element).
	Similarly, since each slice $Y_k\transpose$ has orthonormal rows, matrices in $\{E_{ij} : i,j\in s_k\}$ are linearly independent.
	
	The constraint $\calL_X(A) = 0$ means $\inner{A}{E_{ij}} = 0$ for each $k$ and for each $i,j \in s_k$. To establish the theorem, we need to extract a subset $\calT$ of at least $p(d+1)/2$ of these $qd(d+1)/2$ constraint matrices, and guarantee their linear independence. To this end, let
	\begin{align}
	\calT & = \{ E_{ij} : k \in \{1,\ldots,q\} \textrm{ and } i \in c_k \subseteq s_k, j \in s_k \}.
	\end{align}
	That is, for each slice $k$, $\calT$ includes all constraints of  that slice which involve at least one of the selected rows. For each slice $k$, there are $|c_k|d - \frac{|c_k|(|c_k|-1)}{2}$ such constraints---note the correction for double-counting the $E_{ij}$'s where both $i$ and $j$ are in $c_k$.
	Thus, using $|c_1|+\cdots+|c_q|=p$, the cardinality of $\calT$ is:
	\begin{align}
	|\calT| & = \sum_{k = 1}^q \left[ |c_k|d - \frac{|c_k|(|c_k|-1)}{2} \right] = p(d+1/2) - \frac{1}{2}\sum_{k = 1}^q |c_k|^2.
	\label{eq:contraintcollectionsize}
	\end{align}
	We first show matrices in $\calT$ are linearly independent. Then, we show $|\calT|$ is large enough.
	
	Consider one $E_{ij} \in \calT$: $i,j \in s_k$ for some $k$ and $i = t_\ell$ for some $\ell$ (otherwise, permute $i$ and $j$). By construction of $t$, we can expand $y_j$ in terms of the rows selected in slices 1 to $k$, i.e., $y_j = \sum_{\ell' = 1}^{\ell_k} \alpha_{j,\ell'} y_{t_{\ell'}}$, where $\ell_k = |c_1| + \cdots + |c_k|$. As a result, $E_{ij}$ expands in the basis $\calE$ as follows: $E_{ij} = \sum_{\ell' = 1}^{\ell_k} \alpha_{j,\ell'} E_{t_\ell,t_{\ell'}}$. As noted before, $E_{ij}$'s in $\calT$ contributed by a same slice $k$ are linearly independent. Furthermore, they expand in only a subset of the basis $\calE$, namely, $\calE^{(k)} = \{ E_{t_\ell, t_{\ell'}} : \ell_{k-1} < \ell \leq \ell_k, \ell' \leq \ell_k \}$: $t_\ell$ is a selected row of slice $k$ and $t_{\ell'}$ is a selected row of some slice between 1 and $k$. For $k\neq k'$, $\calE^{(k)}$ and $\calE^{(k')}$ are disjoint; in fact, they form a partition of $\calE$.
	Hence, elements of $\calT$ are linearly independent.
	
	It remains to lower bound~\eqref{eq:contraintcollectionsize}. To this end, use $|c_k| \leq d$ and $|c_1|+\cdots+|c_q|=p$ to get:
	\begin{align*}
	\sum_{k = 1}^q |c_k|^2 \leq \max_{x\in\Rq : \|x\|_\infty \leq d, \|x\|_1 = p} \|x\|^2 = \floorr*{\frac{p}{d}}d^2 + \left( p - \floorr*{\frac{p}{d}}d \right)^2 \leq pd.
	%
	%
	\end{align*}
	Indeed, the maximum in $x$ is attained by making as many of the entries of $x$ as large as possible, that is, by setting $\floor{p/d}$ entries to $d$ and setting one other entry to  $p - \floorr*{p/d}d$ if the latter is nonzero. This many entries are available since $p \leq qd = n$. That this is optimal can be verified using KKT conditions.
	In combination with~\eqref{eq:contraintcollectionsize}, this confirms at least $p(d+1/2)-pd/2 = p(d+1)/2$ linearly independent constraints act on $A$, thus upper bounding $\dim \calF_X$.
	
	To conclude, we argue that the proposed upper bound is essentially tight. Indeed, build $Y\in\calM_p$ by repeating $q$ times the $d$ first rows of $I_p$, then by replacing its $p$ first rows with $I_p$ (to ensure $Y$ has full rank). If $p/d$ is an integer, then exactly the $p/d$ first slices each contribute $d(d+1)/2$ independent constraints, i.e., $\dim \calF_{YY\transpose} = p(p+1)/2 - p(d+1)/2$.
\end{proof}

\section{Equivalence of global non-degeneracy and smoothness} \label{apdx:equivalencenondegeneracysmoothness}

\begin{proof}[Proof of Proposition~\ref{prop:equivalencenondegeneracysmoothness}]
	By Proposition~\ref{prop:assuforallp}, it is sufficient to consider the case $p = n$.
	Consider $X \in \calC$ of rank $r$ and
	a diagonalization $X = QDQ\transpose$, where $D = \diag(\lambda_1, \ldots, \lambda_r, 0, \ldots, 0)$ and $Q = \begin{bmatrix}
	Q_1 & Q_2
	\end{bmatrix}$ is orthogonal of size $n$ with $Q_1 \in \reals^{n\times r}$.
	By~\citep[Thm.~6]{alizadeh1997complementarity}, since $A_1, \ldots, A_m$ are linearly independent, $X$ is primal non-degenerate if and only if the matrices
	\begin{align*}
	B_k & = \begin{bmatrix}
	Q_1\transpose A_k Q_1 & Q_1\transpose A_k Q_2 \\
	Q_2\transpose A_k Q_1 & 0
	\end{bmatrix}, \quad k = 1 \ldots, m
	\end{align*}
	are linearly independent. The $B_k$ are linearly dependent if and only if there exist $\alpha_1, \ldots, \alpha_m$ not all zero such that $\alpha_1 B_1 + \cdots + \alpha_m B_m = 0$. Considering the first $r$ columns of the $B_k$, the latter holds if and only if $\sum_k \alpha_k Q\transpose A_k Q_1 = 0$, which holds if and only if $\sum_k \alpha_k A_k Q_1 = 0$.
	For any $Y\in\Rnp$ such that $X = YY\transpose$, since
	$\spann(Y) = \spann(Q_1)$, we have
	$\sum_k \alpha_k A_k Q_1 = 0$ if and only if $\sum_k \alpha_k A_k Y = 0$. This shows the $B_k$ are linearly dependent if and only if the $A_kY$ are linearly dependent. Thus, $X$ is primal non-degenerate if and only if $\{A_1Y, \ldots, A_mY\}$ are linearly independent. Overall, primal non-degeneracy holds at all $X \in \calC$ if and only if Assumption~\ref{assu:M}a holds.
\end{proof}




 \ack 

NB was partially supported by NSF grant DMS-1719558. Part of this work was done while NB was with the D.I.\ at Ecole normale sup\'erieure de Paris and INRIA's SIERRA team.
ASB was partially supported by NSF grants DMS-1712730 and DMS-1719545. Part of this work was done while ASB was with the Mathematics Department at MIT and partially supported by NSF grant DMS-1317308





\frenchspacing
\bibliographystyle{cpam}
%
\bibliography{../../boumal}

\end{document}